\documentclass[12pt, conference, onecolumn]{IEEEtran}

\usepackage[colorlinks=false,urlcolor=blue,citecolor=blue,linkcolor=blue,bookmarks=true,bookmarksopen=false,pdftitle=created by dvipdf,pdfcreator=NM,pdfauthor=Jain,pdfsubject=mor.sty:6.29.04]{hyperref}
\usepackage{amsmath,amsthm,amstext,amsfonts,amssymb,mathrsfs}
\usepackage{graphicx,subfigure,color,algorithm,algorithmic,fullpage,setspace}
\usepackage{personal}

\IEEEoverridecommandlockouts

\title{
Decentralized Learning for Multi-player Multi-armed Bandits
\thanks{Dileep Kalathil, Naumaan Nayyar and Rahul Jain ({\tt (manisser,nnayyar,rahul.jain) @usc.edu}) are with the Department of Electrical Engineering, University of Southern California, Los Angeles, CA, USA. This research is supported by AFOSR grant FA9550-10-1-0307 and NSF CAREER award CNS-0954116.}\\
\thanks{A preliminary version of this paper is under submission to IEEE CDC 2012. This version contains proofs of all theorems as well as new results on Markovian MABs.}
}
\author{
Dileep Kalathil, Naumaan Nayyar and Rahul Jain 
}

\begin{document}
\maketitle
\date{\today}

\begin{abstract}
We consider the problem of distributed online learning with multiple players in multi-armed bandits (MAB) models. Each player can pick among multiple arms. When a player picks an arm, it gets a reward.  We consider both i.i.d. reward model and Markovian reward model. In the i.i.d. model each arm is modelled as an i.i.d. process  with an unknown distribution with an unknown mean. In the Markovian model, each arm is modelled as a finite, irreducible, aperiodic and reversible Markov chain with an unknown probability transition matrix and stationary distribution. The arms give different rewards to different players. If two players pick the same arm, there is a ``collision'', and neither of them get any reward. There is no dedicated control channel for  coordination or communication among the players. Any other communication between the users is costly and will add to the regret. We propose an online index-based distributed learning policy called ${\tt dUCB_4}$ algorithm that trades off \textit{exploration v. exploitation} in the right way, and achieves expected regret that grows at most as  \textit{near}-$O(\log^2 T)$. The motivation comes from opportunistic spectrum access by multiple secondary users in cognitive radio networks wherein they must pick among various wireless channels that look different to different users. This is the first distributed learning algorithm for multi-player MABs to the best of our knowledge. 
\end{abstract}



\section{Introduction}\label{sec:intro}

In \cite{LaRo85}, Lai and Robbins introduced the classical non-Bayesian multi-armed bandit model. Such models capture the essence of the learning problem that players face in an unknown environment, where the players must not only \textit{explore} to learn but also \textit{exploit} in choosing the best arm. Specifically, suppose a player can choose between $N$ arms. Upon choosing an arm $i$, it gets a reward from a distribution with density $f(x,\theta_i)$. Time is slotted, and players do not know the distributions (nor any statistics about them). The problem is to find a learning policy that minimizes the expected \textit{regret} over some time horizon $T$. It was shown by Lai and Robbins \cite{LaRo85} that there exists an index-type policy that achieves expected regret that grows asymptotically as $\log T$, and this is order-optimal, i.e., there exists no causal policy that can do better. This was generalized by Anantharam, \textit{et al} \cite{AnVaWa87a} to the case of multiple plays, i.e., when the player can pick multiple arms at the same time. In \cite{Ag95},  Agrawal  proposed a sample mean based index policy which achieves $\log T$ regret asymptotically. Assuming that the rewards are coming from a distribution of bounded support, Auer, \textit{et al} \cite{AuCeFi02} proposed a much simpler  sample mean based index policy, called ${\tt UCB_{1}}$, which achieves $\log T$ uniformly over time, not only asymptotically. Also, unlike the policy in \cite{Ag95}, the index doesn't depend on the specific family of distributions that the rewards come from. 

 In \cite{AnVaWa87b}, Anantharam, \textit{et al}  proposed a policy  to the case where the arms are modelled as Markovian, not i.i.d. The rewards are assumed to come from a finite, irreducible and aperiodic Markov chain  represented by a single parameter probability transition matrix. The state of each arm evolves according to an underlying transition probability matrix when the arm is played and remains frozen when passive. Such problems are called \emph{rested Markovian bandit problems} (where \emph{rested} refers to no state evolution until the arm is  played). In \cite{TeLi10}, Tenkin and Liu extended the ${\tt UCB_{1}}$ policy to the case of rested Markovian bandit problems. If some non-trivial bounds on the underlying Markov chains are known a priori, they showed that the policy achieves $\log T$ regret uniformly over time. Also, if no information about the underlying Markov chains is available, the policy can easily be modified to get a \emph{near}-$O(\log T)$ regret asymptotically. The models in which the state of an arm continues to evolve even when it is not played are called \emph{restless Markovian bandit problems}. Restless models are considerably more difficult than the rested models and have been shown to be P-SPACE hard \cite{PaTs99}. This is because the optimal policy no longer will be to ``play the arm with the highest mean reward''. \cite{TeLi11}  employs a weaker notion of regret (\textit{weak regret}) which compares the reward of a policy to that of a policy which always plays the the arm with the highest mean reward. They propose a  policy which achieves $\log T$ (weak) regret uniformly over time if certain bounds on the underlying Markov model are known a priori and achieves a near-$O(\log T)$ (weak) regret asymptotically  when no such knowledge is available. \cite{DaGaKr11} proposes another  simpler policy which achieves the same bounds for weak regret. \cite{LiLiZh11} proposes a policy based on deterministic sequence of exploration and exploitation and achieves the same bounds for weak regret. In \cite{DaGaKrZh11}, the authors consider the notion of \textit{strong regret} and propose a  policy which achieves near-$\log T$ (strong) regret for some special cases of the restless model.


Recently, there is an increasing interest in multi-armed bandit models, partly because of opportunistic spectrum access problems. Consider a user who must choose between $N$ wireless channels. Yet, it knows nothing about the channel statistics, i.e., has no idea of how good or bad the channels are, and what rate it may expect to get from  each channel. The rates could be learnt by exploring various channels. Thus, these have been formulated as multi-armed bandit problems, and index-type policies have been proposed for choosing spectrum channels. In many scenarios, there are multiple users accessing the channels at the same time. Each of these users must be matched to a different channel. These have been formulated as a \textit{combinatorial multi-armed bandit problem} \cite{GaKrJa12} \cite{GaKrLi11}, and it was shown that an ``index-matching'' algorithm that at each instant determines a matching by solving a sum-index maximization problem achieves $O(\log T)$ regret uniformly over time, and this is indeed order-optimal. 

In other settings, the users cannot coordinate, and the problem must be solved in a decentralized manner. Thus, settings where all channels (arms) are identical for all users with i.i.d. rewards have been considered, and index-type policies that can achieve coordination have been proposed that get $O(\log T)$ regret uniformly over time \cite{LiZh10, AnMi11, GaKr11, LiLiZh11}. A similar result for Markovian reward model with weak regret has been shown by \cite{LiLiZh11}, assuming some non-trivial bounds  on the underlying Markov chains are known a priori.  The regret scales only polynomially in the number of users and channels. Surprisingly, the lack of coordination between the players asymptotically imposes no additional cost or regret. 

In this paper, we consider the decentralized multi-armed bandit problem with  distinct arms for each players. We consider both the i.i.d. reward model and the \emph{rested} Markovian reward model. All players together must discover the best arms to play as a team. However, since they are all trying to learn at the same time, they may collide when two or more pick the same arm. We propose an index-type policy ${\tt dUCB_4}$ based on a variation of the ${\tt UCB_1}$ index.
At its' heart is a distributed bipartite matching algorithm such as Bertsekas' auction algorithm \cite{Be92}. This algorithm operates in rounds, and in each round prices for various arms are determined based on bid-values. This imposes communication (and computation) cost on the algorithm that must be accounted for. Nevertheless, we show that when certain non-trivial bounds on the model parameters are known a priori, the ${\tt dUCB_4}$ algorithm that we introduce  achieves (at most) \textit{near}-$O(\log^2 T)$ growth non-asymptotically in expected regret. If no such information  about the model parameters are available, ${\tt dUCB_4}$ algorithm still achieves  (at most) \textit{near}-$O(\log^2 T)$ regret asymptotically. A lower bound, however, is not known at this point, and a work in progress.

The paper is organized as follows. In Section \ref{sec:model}, we present the model and problem formulation. In section \ref{sec:singleplayer} and \ref{sec:singleplayer:M} we present some variations on single player MAB with i.i.d. rewards and Markovian rewards respectively. In section \ref{sec:dmab}, we introduce the decentralized MAB problem with i.i.d. rewards. We then extend the results to the decentralized cases with Markovian rewards in section \ref{sec:dmab:M}. In section \ref{sec:dbm} we present the distributed bipartite matching algorithm which is used in our main algorithm for decentralized MAB. In section \ref{sec:simulations}, we present some simulation results to numerically evaluate the performance of our algorithm.  


\section{Model and Problem Formulation}\label{sec:model}

\subsection{Arms with i.i.d. rewards}
\label{subsec:model:iid} 

We consider an $N$-armed bandit with $M$ players. In a wireless cognitive radio setting \cite{HoBh07}, each arm could correspond to a channel, and each player to a user who wants to use a channel. Time is slotted, and at each instant each player picks an arm.  There is no dedicated control channel for coordination among the players. So, potentially more than one players can pick the same arm at the same instant. We will regard that as a \textit{collision}. Player $i$ playing arm $k$ at time $t$ yields i.i.d. reward $S_{ik}(t) $ with univariate density function $f(s,\theta_{ik})$, where $\theta_{ik}$ is a parameter in the set $\Theta_{ik}$.  We will assume that the rewards are bounded, and without loss of generality lie in $[0,1]$. Let $\m_{i,k}$ denote the mean of $S_{ik}(t)$ w.r.t. the pdf $f(s,\th_{ik})$. We assume that the parameter vector $\th=(\th_{ij}, 1 \leq i \leq M, 1 \leq j \leq N) $ is unknown to the players, i.e., the players have no information about the mean, the distributions or any other statistics about the rewards from various arms other than what they observe while playing. We also assume that each player can only observe the rewards that they get. When there is a collision, we will assume that all players that choose the arm on which there is a collision get  zero reward. This could be relaxed where the players share the reward in some manner though the results do not change appreciably. 

Let $X_{ij}(t)$ be the reward that player $i$ gets from arm $j$ at time $t$. Thus, if player $i$ plays arm $k$ at time $t$ (and there is no collision), $X_{ik}(t) = S_{ik}(t)$, and $X_{ij}(t) = 0,  j \neq k$. Denote the action of player $i$ at time $t$ by $a_i(t) \in \sA := \{1, \ldots, N \}$. Then, the \textit{history} seen by player $i$ at time $t$ is $\sH_i(t) = \{ (a_i(1),X_{i,a_i(1)}(1)),\cdots, (a_i(t),X_{i,a_i(t)}(t))\}$ with $\sH_i(0)=\emptyset$. A \textit{policy} $\alpha_i = (\alpha_i(t))_{t=1}^{\infty}$ for player $i$ is a sequence of maps $\alpha_i(t):\sH_i(t) \to \sA$ that specifies the arm to be played at time $t$ given the history seen by the player.  Let $\sP(N)$ be the set of vectors such that
\begin{eqnarray*}
\sP(N) := \{ \mathbf{a} = (a_{1}, \ldots, a_{M}) : a_{i} \in \sA, a_{i} \neq a_{j}, \text{for}~ i \neq j\}.
\end{eqnarray*}
The players have a \textit{team objective}: namely over a time horizon $T$, they want to maximize the expected sum of rewards $\bbE[\sum_{t=1}^T \sum_{i=1}^{M} X_{i,a_i(t)}(t)]$ over some time horizon $T$.
If the parameters $\mu_{i,j}$ are known, this could easily be achieved by picking a bipartite matching
\begin{equation}
\label{eq:optmatching}
\mathbf{k}^{**} \in \arg \max_{\mathbf{k} \in \sP(N) } \sum_{i=1}^{M} \mu_{i,k_{i}},
\end{equation}
i.e., the optimal bipartite matching with expected reward from each match. Note that this may not be unique. Since the expected rewards, $\mu_{i,j}$, are unknown, the players must pick learning policies that minimize the \textit{expected regret}, defined for policies $\alpha=(\alpha_i, 1 \leq i \leq M  )$ as
\begin{equation}
\label{eq:model:regret}
\sR_{\alpha}(T)=T\sum_i \mu_{i,k_i^{**}}-\bbE_{\alpha}\left[\sum_{t=1}^{T} \sum_{i=1}^{M} X_{i,\alpha_i(t)}(t)\right].
\end{equation}
Our goal is to find a decentralized algorithm that players can use such that together they minimize the expected regret.

\subsection{Arms with Markovian rewards}
\label{subsec:model:m}

Here we follow the model formulation introduced in the previous subsection, with the exception that the rewards are now considered Markovian. The reward that player $i$ gets from arm $j$ (when there is no collision) $X_{ij}$, is modelled as an irreducible, aperiodic, reversible Markov chain on a finite state space $\mathcal{X}^{i,j}$ and represented by a transition probability matrix $P^{i,j}:=\left(p^{i,j}_{x,x^{'}}: x, x^{'} \in \mathcal{X}^{i,j}\right)$. We assume that rewards are bounded and strictly positive, and without loss of generality lie in $(0, 1]$. Let $\mathbf{\pi}^{i,j}:=\left(\pi^{i,j}_{x}, x \in \mathcal{X}^{i,j}\right)$ be the stationary distribution of the Markov chain $P^{i,j}$. The mean reward from arm $j$ for player $i$ is defined as $\mu_{i,j}:=\sum_{x \in \mathcal{X}  ^{i,j} } x \pi^{i,j}_{x} $. Note that the Markov chain represented by $P^{i,j}$ makes a state transition only when player $i$ plays arm $j$. Otherwise it remains \emph{rested}. 

We note that although we use the `big $O$' notation to emphasis the regret order, unless otherwise noted results are non-asymptotic.


\section{Some variations on single player multi-armed bandit with i.i.d. rewards}
\label{sec:singleplayer}

We first present some variations on the single player non-Bayesian multi-armed bandit model. These will prove useful later for the multi-player problem though they should also be of independent interest.

\subsection{${\tt UCB_1}$ with index recomputation every $L$ slots}
Consider the classical single player non-Bayesian $N$-armed bandit problem. At each time $t$, the player picks a particular arm, say $j$, and gets a random reward $X_{j}(t)$. The rewards $X_{j}(t), 1 \leq t \leq T$ are independent and identically distributed according  to some unknown probability measure with an unknown expectation $\mu_{j}$. Without loss of generality, assume that $\mu_{1} > \mu_{i} > \mu_N,$ for  $i=2, \cdots N-1$. Let $n_{j}(t)$ denote the number of times  arm $j$ has been played  by time $t$. Denote $\Delta_{j}:= \mu_{1}-\mu_{j}$, $\Delta_{min}:=\min_{j, j \neq 1} \Delta_{j}$ and $\Delta_{max}:=\max_{j} \Delta_{j}$.  The regret for any policy $\alpha$ is 
\begin{equation}
\sR_{\alpha}(T) := \mu_{1} T - \sum_{j=1}^{N} \mu_{j} \bbE_{\alpha}[n_{j}(T)].
\end{equation}
${\tt UCB_1}$ index \cite{AuCeFi02} is defined as 
\begin{equation} 
g_{j}(t) := \overline{X}_{j}(t)+\sqrt{\frac{2 \log(t)} {n_{j}(t)}},
\end{equation}
where $\overline{X}_{j}(t)$ is the average reward obtained by playing arm $j$ by time $t$. It is defined as $\overline{X}_{j}(t) = \sum_{m=1}^{t} r_{j}(m)/n_{j}(t)$, where $r_{j}(m)$ is the reward obtained from arm $j$ at time $m$. If the arm $j$ is played at time $t$ then $r_{j}(m)=X_{j}(m)$ and otherwise $r_{j}(t)=0$. 
Now, an index-based policy called ${\tt UCB_1}$ \cite{AuCeFi02} is to pick the arm that has the highest index at each instant. It can be shown that this algorithm achieves regret that grows logarithmically in $T$ non-asymptotically. 

An easy variation of the above algorithm which will be useful in our analysis of subsequent algorithms is the following. Suppose the index is re-computed only once every $L$ slots. In that case, it is easy to establish the following.


\begin{theorem}
\label{thm:ucb1}
Under the ${\tt UCB_1}$ algorithm with recomputation of the index once every $L$ slots, the expected regret by time $T$ is given by
\begin{equation}
\sR_{\tt UCB_1}(T) \leq \sum_{j>1}^{N} \frac{8 L\log T}{\Delta_{j} } + L\left(1 + \frac{\pi^{2} }{3}\right) \sum_{j>1}^{N} \Delta_{j}.
\end{equation}
\end{theorem}
%

The proof follows \cite{AuCeFi02} and taking into account the fact that every time a suboptimal arm is selected, it is played for the next $L$ time slots.  We omit it due to space consideration.

\subsection{${\tt UCB_4}$ Algorithm when index computation is costly}

Often, learning algorithms pay a penalty or cost for computation. This is particularly the case when the algorithms must solve combinatorial optimization problems that are NP-hard. Such costs also arise in decentralized settings wherein algorithms pay a communication cost for coordination between the decentralized players. This is indeed the case, as we shall see later when we present an algorithm to solve the decentralized multi-armed bandit problem. Here, however, we will just consider an ``abstract'' communication or computation cost. The problem we formulate below can be solved with better regret bounds than what we present. At this time though we are unable to design algorithms with better regret bounds, that also help in decentralization.

Consider a computation cost every time the index is recomputed. Let the cost be $C$ units. Let $m(t)$ denote the number of times the index is computed by time $t$. Then, under policy $\alpha$ the expected regret is now given by
\begin{equation}
\label{eq:regret_singleuser}
\tilde{\sR}_{\alpha}(T) := \mu_{1}T - \sum_{j=1}^{N} \mu_{j} \bbE_{\alpha}[n_{j}(T)] + C\bbE_{\alpha}[m(T)].
\end{equation}
It is easy to argue that the ${\tt UCB_1}$ algorithm will give a regret $\Omega (T)$ for this problem. We present an alternative algorithm called ${\tt UCB_4}$ algorithm, that gives sub-linear regret. Define the ${\tt UCB_4}$ index  
\begin{equation}
g_{j}(t):=\overline{X}_{j}(t)+\sqrt{\frac{3 \log(t)} {n_{j}(t)}}.
\end{equation}
We define an arm $j^*(t)$ to be the \textit{best arm} if \( j^{*}(t) \in \arg \max_{1 \leq i \leq N } g_{i}(t).\)

\begin{algorithm}
\caption{: $\tt UCB_4$}
\label{algo:ucb4}
{\small
\begin{algorithmic}[1]
\STATE {\bf Initialization:} Select each arm $j$ once for $t \leq N$. Update the $\tt UCB_4$ indices. Set $\eta=1$.
\WHILE {($t \leq T$)}
\IF {($\eta = 2^{p}$ for some $p = 0, 1, 2, \cdots$)}
\STATE Update the index vector $g(t)$;
\STATE Compute the \textit{best arm} $ j^{*}(t)$;
\IF{$(j^*(t) \neq j^*(t-1))$}
\STATE Reset $\eta=1$;
\ENDIF
\ELSE 
\STATE $j^*(t) = j^*(t-1)$;
\ENDIF
\STATE Play arm $j^*(t)$;
\STATE Increment counter $\eta=\eta+1$; $t=t+1$;
\ENDWHILE
\end{algorithmic}
}
\end{algorithm}

We will use the following concentration inequality.\\
\noindent{{\bf Fact 1:} Chernoff-Hoeffding inequality} \cite{Po84}\\
Let $X_{1}, \ldots, X_{t}$ be random variables with a common range such that $\bbE[X_{t}|X_{1}, \ldots, X_{t-1}]=\mu$. Let $S_{t}=\sum_{i=1}^{t} X_{i}$. Then for all $a \geq 0$, 
\begin{eqnarray}
\bbP\left(S_{t} \geq t\mu + a\right) \leq e^{-2a^{2}/t}, ~~\text{and}~~  \bbP\left(S_{t} \leq t\mu - a\right)  \leq e^{-2a^{2}/t}.
\end{eqnarray} 
\begin{theorem}
\label{thm:ucb4_regret}
The expected regret for the single player multi-armed bandit problem with per computation cost $C$ using the ${\tt UCB_4}$ algorithm is given by 
\begin{eqnarray*}
\tilde{\sR}_{\tt UCB_4}(T)   &\leq & (\Delta_{max}+C (1+\log T)) \cdot \left(\sum_{j>1}^{N} \frac{12 \log T}{\Delta_{j}^{2} } + 2N\right).
\end{eqnarray*}
Thus, \( \tilde{\sR}_{\tt UCB_4}(T) = O(\log^{2} T).\)
\end{theorem}

\begin{proof}
We prove this in two steps. First, we compute the expected number of times a suboptimal arm is played and then the expected number of times we recompute the index.

Consider any suboptimal arm $j>1$. Denote $c_{t,s} = \sqrt{3 \log t /s}$ and the indicator function of the event $A$ by $I\{A\}$. let $\t_{j,m}$ be the time at which the player makes the $m$th transition to the arm $j$ from  another arm and $\t_{j,m}^{'}$ be the time at which the player makes the $m$th transition from the arm $j$ to another arm. Let $\tilde{\tau}_{j,m}^{'} = \min\{\t_{j,m}^{'}, T\}$. Then,\\ 
$n_{j}(T) \leq  1 + \sum_{m=1}^{T}  (\tilde{\tau}_{j,m}^{'}-\tau_{j,m}) I\{\text{Arm} ~j~\text{is picked at time}~  \tau_{j,m}, \tau_{j,m} \leq T \} $
\begin{eqnarray}
\label{eq:singleusr_njbd_st1} 
&\leq & 1 + \sum_{m=1}^{T}  (\tilde{\tau}_{j,m}^{'}-\tau_{j,m}) I\{ g_{j}(\tau_{j,m} -1) \geq  g_{1}(\tau_{j,m}-1), \tau_{j,m} \leq T  \} \nonumber \\
&\leq & l + \sum_{m=1}^{T}  (\tilde{\tau}_{j,m}^{'}-\tau_{j,m}) I\{ g_{j}(\tau_{j,m} -1) \geq  g_{1}(\tau_{j,m}-1), \tau_{j,m} \leq T, n_{j}(\tau_{j,m}-1) \geq l  \} \nonumber \\
& \stackrel{(a)}{\leq}&   l + \sum_{m=1}^{T} \sum_{p=0}^{\infty} 2^{p} I\{ g_{j}(\tau_{j,m} +2^{p}-2) \geq   g_{1}(\tau_{j,m} +2^{p}-2), \tau_{j,m} + 2^{p} \leq T,   n_{j}(\tau_{j,m}-1) \geq l\} \nonumber \\
& \stackrel{(b)}{\leq}&  l + \sum_{m=2}^{T} \sum_{p=0}^{\infty} 2^{p} I\{ g_{j}(m +2^{p}-2) \geq   g_{1}(m +2^{p}-2), m+2^{p} \leq T,   n_{j}(m-1) \geq l\} \nonumber \\
& \leq & l + \sum_{m=1}^T \sum_{p \geq 0, m+2^{p} \leq T} 2^{p} I\{ \overline{X}_{j}(m+2^{p}-1)  + c_{m+2^{p}-1, n_{j}(m+2^{p}-1)}  \geq \nonumber \\
~~&& \hspace{4cm} \overline{X}_{1}(m+2^{p}-1)  + c_{m+2^{p}-1, n_{1}(m+2^{p}-1)}, n_{j}(m-1) \geq l  \}\\ \nonumber 
& \leq & l + \sum_{m=1}^T \sum_{p \geq 0, m+2^{p} \leq T} 2^{p}  I\{ \max_{l \leq s_{j} < m+2^{p} }   \overline{X}_{j}(m+2^{p}-1)+c_{m+2^{p}-1, s_{j}}  \geq \nonumber \\
~~&& \hspace{4cm}  \min_{1 \leq s_{1} < m+2^{p} }   \overline{X}_{1}(m+2^{p}-1)+c_{m+2^{p}-1, s_{1}} \} \nonumber \\
&\leq & l + \sum_{m=1}^{\infty } \sum_{p \geq 0, m+2^{p} \leq T} 2^{p} \sum_{s_{1}=1}^{m+2^{p}} \sum_{s_{j}=l}^{m+2^{p}}   I\{\overline{X}_{j}(m+2^{p}) +c_{m+2^{p}, s_{j}}  \geq  \overline{X}_{1}(m+2^{p})+c_{m+2^{p}, s_{1}} \}.
\end{eqnarray}
In Algorithm \ref{algo:ucb4} ($\tt UCB_{4}$), if an arm is for the $p$th time consecutively (without switching to any other arms in between), it is be played for the next $2^{p}$ slots. Inequality (a) uses this fact. In the inequality (b), we replace $\tau_{j,m}$ by $m$ which is clearly an upper bound.  
Now, observe that the event $\{  \overline{X}_{j}(m+2^{p})+c_{m+2^{p}, s_{j}}  \geq \overline{X}_{1}(m+2^{p})+c_{m+2^{p}, s_{1}} \}$ implies at least one of the following events,
\begin{eqnarray}
\label{events_singleuser}
A:=\big\{ \overline{X}_{1}(m+2^{p}) &\leq & \mu_{1} - c_{m+2^{p}, s_{1}} \big\} , \hspace{0.5cm}  B:=\big\{\overline{X}_{j}(m+2^{p}) \geq  \mu_{j} + c_{m+2^{p}, s_{j}}\big\}, \nonumber  \\
&& \text{or}~ C:=\big\{\mu_{1} < \mu_{j} + 2 c_{m+2^{p}, s_{j}}\big\}. 
\end{eqnarray}
Now, using the Chernoff-Hoeffding bound, we get
\begin{eqnarray*}
\bbP\left(\overline{X}_{1}(m+2^{p}) \leq  \mu_{1} - c_{m+2^{p}, s_{1}}\right) \leq (m+2^{p})^{-6},~\bbP\left(\overline{X}_{j}(m+2^{p}) \geq  \mu_{j} + c_{m+2^{p}, s_{j}}\right) \leq (m+2^{p})^{-6}.
\end{eqnarray*}
For $l = \left \lceil \frac{ 12 \log T}{ \Delta^{2}_{j} } \right \rceil $, the last event in (\ref{events_singleuser}) is false. In fact, $\mu_{1} - \mu_{j} - 2 c_{m+2^{p}, s_{j}}$
\begin{eqnarray}
= \mu_{1} - \mu_{j} - 2 \sqrt{3 \log(m+2^{p})/s_{j}} && \geq \mu_{1} - \mu_{j} - \Delta_{j} = 0,~ \text{for}~ s_{j} \geq \left \lceil 12 \log T/\Delta_{j}^{2} \right \rceil. \nonumber \\
\text{So, we get,}~ \bbE[n_{j}(T)] \leq  \left \lceil 12 \log T/\Delta_{j}^{2} \right \rceil  &+& \sum_{m=1}^{\infty } \sum_{p=0}^{\infty} 2^{p} \sum_{s_{1}=1}^{m+2^{p}} \sum_{s_{j}=1}^{m+2^{p}} 2 (m+2^{p})^{-6} \nonumber \\
\leq  \left \lceil 12 \log T/\Delta_{j}^{2} \right \rceil &+& 2 \sum_{m=1}^{\infty } \sum_{p=0}^{\infty} 2^{p}  (m+2^{p})^{-4}  \leq \frac{ 12 \log T}{ \Delta^{2}_{j} } + 2.\label{eq:ucb4:nj}
\end{eqnarray}

Next, we upper-bound the expectation of  $m(T)$, the number of index computations performed by time $T$. We can write $m(T) = m_1(T) + m_2(T)$, where $m_1(T)$ is the number of index updates that result in an optimal allocation, and $m_2(T)$ is the number of index updates that result in a suboptimal allocation. Clearly, the number of updates resulting in a suboptimal allocation is less than the number of times a suboptimal arm is played. Thus, 
\begin{equation}
\label{eq:ucb4:m2}
\bbE [m_2(T)] \leq \sum_{j>1}^N\bbE[n_{j}(T)].
\end{equation}
To bound $\bbE [m_1(T)]$, let $\t_l$ be the time at which the player makes the $l$th transition to an optimal  arm from a suboptimal arm and $\t_l'$ be the time at which the player makes the $l$th transition from an optimal arm to a suboptimal arm. Then, $ m_1(T) \leq \sum_{l=1}^{n_{sub}(T)} \log | \t_l - \t_l' |$, 
where $n_{sub}(T)$ is the total number of such transitions by time $T$. Clearly, $n_{sub}(T)$ is upper-bounded by the total number of times the player picks a sub-optimal arm. Also, $\log | \t_l - \t_l' | \leq \log T$.   So, 
\begin{equation}
\label{eq:ucb4:m1}
\bbE [m_1(T)] \leq \sum_{j>1}^{N} \bbE [n_j(T)] \cdot \log T.
\end{equation}
Thus, from bounds (\ref{eq:ucb4:m2}) and (\ref{eq:ucb4:m1}), we get
\begin{equation}
\label{eq:ucb4:m}
\bbE [m(T)] \leq \sum_{j>1}^{N} \bbE [n_{j}(T)] \cdot (1+\log T).
\end{equation}
Now, using equation (\ref{eq:regret_singleuser}), the expected regret is
\begin{eqnarray*}
\tilde{\sR}_{\tt UCB_4}(T) &=& \sum_{j>1}^{N} \bbE [n_j(T)] \cdot \Delta_{j} + C \bbE[m(T)] \leq \Delta_{max} \sum_{j>1}^{N} \bbE [n_j(T)] + C \bbE[m(T)] \\
&\leq & \left( \Delta_{max}  + C (1+\log T) \right) \sum_{j>1}^{N} \bbE [n_j(T)].
\end{eqnarray*}
by using  (\ref{eq:ucb4:m}). Now, by bound (\ref{eq:ucb4:nj}), we get the desired bound on the expected regret.
\end{proof}

\noindent \textbf{Remarks.} 1. It is easy to show that the lower bound for the single player MAB problem with computation costs is $ \Omega (\log T)$. This can be achieved by the ${\tt UCB_2}$ algorithm \cite{AuCeFi02}. To see this, note that the number of times the player selects a suboptimal arm when using ${\tt UCB_2}$ is $O (\log T)$. Since $\bbE [n_j(T)] = O (\log T)$,  we get
\( \bbE [ \sum_{j>1}^{N} n_j(T) ] =  O(\log T), \)
and also \( \bbE [m_2(T)]  = O(\log T). \)
Now, since the epochs are not getting reset after every switch and are exponentially spaced,  the number of updates that result in the optimal allocation, \( m_1(T) \leq \log T. \) These together yield
\begin{equation*}
\tilde{\sR}_{\tt UCB_2}(T)  \leq \sum_{j>1}^{N} \bbE [n_j(T)] \cdot \Delta_j + C \bbE[m(T)] = O(\log T).
\end{equation*}
2. Variations of the ${\tt UCB_2}$ algorithm  that use a deterministic schedule can also be used \cite{LiZh11}. But it is unknown at this time if these can be used in solving the decentralized MAB problem that we introduce in the next section. This is the main reason for introducing the ${\tt UCB_4}$ algorithm.

\subsection{Algorithms with finite precision indices}
\label{subsec:limitedprec}
Often, there might be a cost to compute the indices to a particular precision. In that case, indices may be known upto some $\e$ precision, and it may not possible to tell which of two indices is greater if they are within $\e$ of each other. The question then is how is the performance of various index-based policies such as ${\tt UCB_1, UCB_4}$, etc. affected if there are limits on index resolution, and only an arm with an $\e$-highest index can be picked. We first show that if $\Delta_{min}$ is known, we can fix a precision $0 < \epsilon < \Delta_{min}$, so that ${\tt UCB_4}$ algorithm will achieve order log-squared regret growth with $T$. If $\Delta_{min}$ is not known, we can pick a positive monotone sequence $\{\epsilon_{t}\}$  such that $\e_t \to 0$, as $t \to \infty$. Denote the cost of computation for $\e$-precision be $C(\e)$. We assume that  $C(\e) \rightarrow  \infty $ monotonically as $\epsilon \rightarrow 0$. 

\begin{theorem}\label{thm:ucb4_cmin}
(i) If $\Delta_{min}$ is known, choose an $0 < \e < \D_{min}$. Then, the expected regret of the ${\tt UCB_4}$ algorithm with $\e$-precise computations is given by
\begin{eqnarray*}
\tilde{\sR}_{\tt UCB_4}(T) &\leq&  \left(\Delta_{max}+C(\epsilon) (1+ \log T )\right) \cdot  \left( \sum_{j>1}^{N} \frac{12 \log T}{(\Delta_{j} - \epsilon)^{2} } + 2N\right).
\end{eqnarray*}
Thus, $\tilde{\sR}_{\tt UCB_4}(T) = O(\log^{2} T).$\\
(ii) If $\Delta_{min}$ is  unknown, denote $\e_{min}=\D_{min}/2$ and choose a positive monotone sequence $\{\e_t\}$ such that $\e_t \to 0$ as $t \to \infty$. Then, there exists a $t_0 > 0$ such that for all $T > t_0$,
\begin{eqnarray*}
\tilde{\sR}_{\tt UCB_4}(T)   &\leq&   \left( \Delta_{max} + C(\epsilon_{min}) \right) t_{0}  +  (\Delta_{max}+C(\epsilon_{T}) (1+ \log T )) \cdot  \left( \sum_{j>1}^{N} \frac{12 \log T}{(\Delta_{j} - \epsilon_{min})^{2} } + 2N\right)
\end{eqnarray*}
where $t_{0}$ is the smallest $t$ such that $\epsilon_{t_{0}} < \epsilon_{min}$. Thus by choosing an arbitrarily slowly increasing sequence $\{\epsilon_{t}\}$, we can make the regret arbitrarily close to $O(\log^{2} T)$ asymptotically. 
\end{theorem}
\begin{proof} 
(i) The proof is only a slight modification of the proof given in Theorem \ref{thm:ucb4_regret}. Due to the $\epsilon$ precision, the player will pick a suboptimal arm if the event $\{ \overline{X}_{j}(m+2^{p})+c_{m+2^{p}, s_{j}} + \epsilon  \geq \overline{X}_{1}(m+2^{p})+c_{m+2^{p}, s_{1}} \}$ occurs.  Thus equation (\ref{eq:singleusr_njbd_st1}) becomes, $n_{j}(T)$ \[\leq  l + \sum_{m=1}^{\infty } \sum_{p \geq 0, m+2^{p} \leq T} 2^{p} \sum_{s_{1}=1}^{m+2^{p}} \sum_{s_{j}=l}^{m+2^{p}}   I\{\overline{X}_{j}(m+2^{p}) +c_{m+2^{p}, s_{j}} + \epsilon \geq  \overline{X}_{1}(m+2^{p})+c_{m+2^{p}, s_{1}} \}.\]
Now, the event $\{ \overline{X}_{j}(m+2^{p})+c_{m+2^{p}, s_{j}} + \epsilon  \geq \overline{X}_{1}(m+2^{p})+c_{m+2^{p}, s_{1}} \}$  implies that at least one of the following events must occur:
\begin{eqnarray}
\label{events_singleuser_finres}
A:= \big\{\overline{X}_{1}(m+2^{p}) \leq \mu_{1} - c_{m+2^{p}, s_{1}}\big\} , && B:=\big\{ \overline{X}_{j}(m+2^{p})  \geq \mu_{j} + \epsilon + c_{m+2^{p}, s_{j}}\big\}, \nonumber \\
C:=\big\{ \mu_{1} < \mu_{j} + \epsilon + 2 c_{m+2^{p}, s_{j}}\big\}, && \text{or}~ D:=\big\{ \mu_{1} < \mu_{j} + \epsilon\big\}.
\end{eqnarray}
Since $\{ \overline{X}_{j}(m+2^{p})  \geq \mu_{j} + \epsilon + c_{m+2^{p}, s_{j}} \} \subseteq \{ \overline{X}_{j}(m+2^{p})  \geq \mu_{j} +  c_{m+2^{p}, s_{j}} \}$, we have
\begin{eqnarray*}
\bbP  (\{ \overline{X}_{j}(m+2^{p})  &\geq& \mu_{j} + \epsilon + c_{m+2^{p}, s_{j}} \}) \leq \bbP (\{ \overline{X}_{j}(m+2^{p})  \geq \mu_{j} +  c_{m+2^{p}, s_{j}} \}).
\end{eqnarray*}

Also, for $l = \left \lceil  12 \log T/(\Delta_{j} - \epsilon)^{2}\right \rceil $, the event $C$ cannot happen. In fact, 
$\mu_{1} - \mu_{j} - \epsilon - 2 c_{t+2^{p}, s_{j}} = \mu_{1} - \mu_{j} - \epsilon  - 2 \sqrt{\frac{3 \log (t+2^{p})} {s_{j}}} \geq  \mu_{1} - \mu_{j} - \epsilon - (\Delta_{j}-\epsilon) = 0,$
for $s_{j} \geq \left \lceil  12 \log T/(\Delta_{j} - \epsilon)^{2}\right \rceil$.  If $\epsilon < \Delta_{min}$, the last event (D) in equation (\ref{events_singleuser_finres}) is also not true.  Thus, for  $0 < \e < \Delta_{min}$, we get
 \begin{equation}
 \label{eq:ucb4_finres}
\mathbb{E}[n_j(T)] \leq \frac{12\log(n)}{(\Delta_{j} - \epsilon)^2} + 2.
\end{equation} The rest of the proof is the same as in Theorem \ref{thm:ucb4_regret}. 
Now, if $\Delta_{min}$ is known, we can choose $0 < \epsilon < \Delta_{min}$ and by Theorem \ref{thm:ucb4_regret} and bound (\ref{eq:ucb4_finres}), we get the desired result. 

\noindent (ii) If $\D_{min}$ is unknown, we can choose a positive monotone sequence $\{\e_t\}$ such that $\e_t \to 0$ as $t \to \infty$. Thus, there exists a $t_0$ such that for $t > t_{0}$, $\e_t < \epsilon_{min}$. We may get a linear regret upto time $t_0$ but after that the analysis follows that in the proof of Theorem \ref{thm:ucb4_regret}, and regret  grows only sub-linearly. Since $C(\cdot )$ is monotone, $C(\epsilon_{T}) > C(\epsilon_{t})$ for all $t < T$. The last part can now be trivially established using the obtained bound on the expected regret. 
\end{proof}

\section{Single Player Multi-armed Bandit with Markovian Rewards}
\label{sec:singleplayer:M}
Now, we consider the scenario where the rewards obtained from an arm are not i.i.d. but come from a Markov chain. Reward from each arm is modelled as an irreducible, aperiodic, reversible Markov chain on a finite state space $\mathcal{X}^{i}$ and represented by a transition probability matrix $P^{i} := \left( p^{i}_{x,x^{'}} : x, x^{'} \in \mathcal{X}^{i} \right) $.  Assume that the reward space $\mathcal{X}^{i} \subseteq (0, 1]$. Let  $X_{i}(1), X_{i}(2), \ldots$ denote the successive rewards from arm $i$.  All arms are mutually independent.  Let $\mathbf{\pi^{i}}:=\left(\pi^{i}_{x}, x \in \mathcal{X}^{i}\right)$ be the stationary distribution of the Markov chain $P^{i}$. Since the Markov chains are ergodic under these assumptions, the mean reward from arm $i$ is given by $\mu_{i} := \sum_{x \in \mathcal{X}  ^{i} } x \pi^{i}_{x} $. 
Without loss of generality, assume that $\mu_{1} > \mu_{i} > \mu_N,$ for  $i=2, \cdots N-1$. As before, $n_{j}(t)$ denotes the number of times arm $j$ has been played by time $t$. Denote $\Delta_{j}:=\mu_{1}-\mu_{j}$, $\Delta_{min}:=\min_{j, j \neq 1} \Delta_{j}$ and $\Delta_{max}:=\max_{j} \Delta_{j}$.  Denote $\pi_{min}:=\min_{1 \leq i \leq N, x \in \mathcal{X}^{i}} \pi^{i}_{x} $, $x_{max}:=\max_{1 \leq i \leq N, x \in \mathcal{X}^{i} } x$ and $x_{min}:=\min_{1 \leq i \leq N, x \in \mathcal{X}^{i} } x$. Let ${\hat{\pi}^{i}}_{x}:=\max\{\pi^{i}_{x}, 1-\pi^{i}_{x} \}$ and $\hat{\pi}_{max}:= \max_{1 \leq i \leq N, x \in \mathcal{X}^{i} } {\hat{\pi}^{i}}_{x} $. Let  $|\mathcal{X}^{i}|$ denote the cardinality of the state space $\mathcal{X}^{i}$, $|\mathcal{X}|_{max}:=\max_{1 \leq i \leq N } |\mathcal{X}^{i}|$. Let $\rho^{i}$ be the  eigenvalue gap, $1 - \lambda_{2}$, where $\lambda_{2}$ is the second largest eigenvalue of the matrix ${P^{i}}^{2}$. Denote $\rho_{max}:= \max_{1 \leq i \leq N } \rho^{i}$ and $\rho_{min}:=\min_{1 \leq i \leq N } \rho^{i}$, where $\rho^{i}$ is the eigenvalue gap of the $i$th arm. 

The total reward  obtained by the time $T$ is then given by $S_{T}= \sum_{j=1}^{N} \sum_{s=1}^{n_{j}(T)} X_{j}(s)$. The regret for any policy $\alpha$ is defined as  
\begin{equation}
\label{eq:regret_singleuser_M} 
\tilde{\sR}_{M, \alpha}(T) := \mu_{1} T - \bbE_{\alpha } \sum_{j=1}^{N} \sum_{s=1}^{n_{j}(T)} X_{j}(s) + C \bbE_{\alpha }[m(T)]
\end{equation} where $C$ is the cost per computation and $m(T)$ is the number of times the index is computed by time $T$, as described in section \ref{sec:singleplayer}. Define the index 
\begin{equation} 
\label{markov_single_index}
g_{j}(t) := \overline{X}_{j}(t)+\sqrt{\frac{ \kappa \log(t)} {n_{j}(t)}},
\end{equation}
where $\overline{X}_{j}(t)$ is the average reward obtained by playing arm $j$ by time $t$, as defined in the previous section. $\kappa$ can be any constant satisfying $\kappa > 168 |\mathcal{X}|_{max}^{2}/\rho_{min}$.

 We introduce one more notation here. If $\mathcal{F}$ and $\mathcal{G}$ are two $\sigma$-algebras, then $\mathcal{F} \vee \mathcal{G}$ denotes the smallest $\sigma$-algebra containing $ \mathcal{F}$ and  $\mathcal{G}$. Similarly, if $\{\mathcal{F}_{t}, t=1, 2, \ldots\}$ is a collection of $\sigma$-algebras, then $\vee_{t \geq 1 } F_{t}$ denotes the smallest $\sigma-$algebra containing $\mathcal{F}_{1}, \mathcal{F}_{2}, \ldots$
 
The following can be  derived easily from Lemma \ref{lemma_AnVa1} \cite{AnVaWa87b}, reproduced in the appendix. 
\begin{lemma}
\label{lemma_AnVa2}
If the reward of each arm is given by a Markov chain satisfying the hypothesis of Lemma \ref{lemma_AnVa1}, then under any policy $\alpha$ we have
\begin{equation}
\label{eq:lemma_TeLi1}
\tilde{\sR}_{M, \alpha}(T) \leq  \sum_{j=2}^{N} \Delta_{j} \bbE_{\alpha }[n_{j}(T)] + K_{\mathcal{X}, P} + C \bbE_{\alpha }[m(T)]
\end{equation}
where $K_{\mathcal{X}, P}=\sum_{j=1}^{N}  \sum_{x \in \mathcal{X}^{j}} x/\pi^{j}_{min}$ and $\pi^{j}_{min} = \min_{x \in \mathcal{X}^{j} } \pi^{j}_{x}$
\end{lemma}
\begin{proof}
Let  $X_{j}(1), X_{j}(2), \ldots$ denote the successive rewards from arm $j$.  Let $\mathcal{F}^{j}_{t}$ denotes the $\sigma$-algebra generated by $\left(X_{j}(1), \ldots, X_{j}(t)\right)$. Let $\mathcal{F}^{j}=\vee_{t \geq 1}\mathcal{F}^{j}_{t}$ and $\mathcal{G}^{j}=\vee_{i \neq j } F^{i}$.  Since arms are independent, $\mathcal{G}^{j}$ is independent of $\mathcal{F}^{j}$. Clearly, $n_{j}(T)$ is a stopping time with respect to $\mathcal{G}^{j} \vee \mathcal{F}^{j}_{T}$. 
The total reward is $S_{T} =  \sum_{j=1}^{N} \sum_{s=1}^{n_{j}(T)} X_{j}(s) =  \sum_{j=1}^{N}  \sum_{x \in \mathcal{X}^{j}} x N(x, n_{j}(T) )$ where $N(x, n_{j}(T)) := \sum_{t=1}^{n_{j}(T)} I\{X_{j}(t)=x\} $.
Taking the expectation and using the Lemma \ref{lemma_AnVa1}, we have  $ \left \lvert \bbE[S_{T}] - \sum_{j=1}^{N}  \sum_{x \in \mathcal{X}^{j}} x \pi^{j}_{x} \bbE[n_{j}(T)] \right \rvert  \leq \sum_{j=1}^{N}  \sum_{x \in \mathcal{X}^{j}} x/\pi^{j}_{min}$, 
which implies \\ $\left \lvert \bbE[S_{T}] - \sum_{j=1}^{N}  \mu_{j} \bbE[n_{j}(T)] \right \rvert  \leq  K_{\mathcal{X}, P}, $
where $K_{\mathcal{X}, P} = \sum_{j=1}^{N}  \sum_{x \in \mathcal{X}^{j}} x/\pi^{j}_{min}$. Since regret  \\$\tilde{\sR}_{M, \alpha}(T) = \mu_{1} T - \bbE_{\alpha } \sum_{j=1}^{N} \sum_{t=1}^{n_{j}(T)} X_{j}(t) + C \bbE_{\alpha }[m(T)]$ (c.f. equation (\ref{eq:regret_singleuser_M})), we get
\begin{equation*}
|\tilde{\sR}_{M, \alpha}(T) - \left( \mu_{1} T - \sum_{j=1}^{N}  \mu_{j} \bbE[n_{j}(T)] + C \bbE_{\alpha}[m(T)] \right)| \leq K_{\mathcal{X},P}.
\end{equation*}
\end{proof}
We will use a concentration inequality for Markov chains (Lemma \ref{lemma_Le}, from \cite{Le98}), reproduced in the appendix.

\begin{theorem}
\label{thm:ucb4_regret_M}
(i) If  $|\mathcal{X}|_{max}$ and $\rho_{min}$ are known, choose $\kappa > 168 |\mathcal{X}|_{max}^{2}/\rho_{min}$. Then, the expected regret using the ${\tt UCB_4}$ algorithm with the index defined as in (\ref{markov_single_index})  for the single player multi-armed bandit problem with Markovian rewards and  per computation cost $C$  is given by 
\begin{eqnarray*}
\tilde{\sR}_{M, \tt UCB_4}(T)   &\leq & (\Delta_{max}+C (1+\log T)) \cdot \left(\sum_{j>1}^{N} \frac{4 \kappa \log T}{\Delta_{j}^{2} } + N (2D+1) \right) +  K_{\mathcal{X},P}
\end{eqnarray*}
where $D=\frac{|\mathcal{X}|_{max}}{\pi_{min} } $.
Thus, \( \tilde{\sR}_{M, \tt UCB_4}(T) = O(\log^{2} T).\) \\
(ii) If $|\mathcal{X}|_{max}$ and $\rho_{min}$ are not known, choose a positive monotone sequence $\{\kappa_{t}\}$ such that $\kappa_{t} \rightarrow \infty $ as $t \rightarrow \infty$ and $\kappa_{t} \leq t$.  Then, $\tilde{\sR}_{M, \tt UCB_4}(T) = O( \kappa_{T} \log^{2} T)$. Thus, by choosing an arbitrarily slowly increasing sequence  $\{\kappa_{t}\}$ we can make the regret arbitrarily close to  $\log^{2} T$.
\end{theorem}
\begin{proof}
(i) Consider any suboptimal arm $j>1$. Denote $c_{t,s} = \sqrt{\kappa \log t /s}$. As in the proof of  Theorem \ref{thm:ucb4_regret}, we start by bounding $n_{j}(T)$. The initial steps are the same as in the proof of Theorem \ref{thm:ucb4_regret}. So,  we skip those steps and start from the inequality (\ref{eq:singleusr_njbd_st1}) there.  
\begin{eqnarray*}
n_{j}(T) &\leq & l + \sum_{m=1}^{\infty } \sum_{p \geq 0, m+2^{p} \leq T} 2^{p} \sum_{s_{1}=1}^{m+2^{p}} \sum_{s_{j}=l}^{m+2^{p}}  I\{\overline{X}_{j}(m+2^{p}) +c_{m+2^{p}, s_{j}}  \geq  \overline{X}_{1}(m+2^{p})+c_{m+2^{p}, s_{1}} \}.
\end{eqnarray*}
The event $\{\overline{X}_{j}(m+2^{p}) +c_{m+2^{p}, s_{j}}  \geq  \overline{X}_{1}(m+2^{p})+c_{m+2^{p}, s_{1}} \}$ is true only if at least one of the events shown in display (\ref{events_singleuser}) are true.
We note that, for any initial distribution $\lambda^{j}$ for arm $j$, \\

\begin{equation}
N_{\lambda^{j}}   = \left \lVert \left(\frac{\lambda^{j}_{x}}{\pi^{j}_{x}}, x \in \mathcal{X}^{j} \right) \right \rVert  _{2} \leq \sum_{x \in \mathcal{X}^{j}} \left \lVert \left(\frac{\lambda^{j}_{x}}{\pi^{j}_{x}}\right)  \right \rVert _{2} \leq \frac{1}{\pi_{min}}.
\end{equation}
Also, $x_{max} \leq 1$. Let $n^{j}_{x}(s_{j})$ be the number of times the state $x$ is observed when arm $j$ is pulled $s_{j}$ times. Then, the probability of the first event in (\ref{events_singleuser}), \\
$\bbP(\overline{X}_{j}(m+2^{p})  \geq \mu_{j} + c_{m+2^{p}, s_{j}})$
\begin{eqnarray*}
&& = \bbP\left( \sum_{x \in \mathcal{X}^{j}} x n^{j}_{x}(s_{j}) \geq s_{j} \sum_{x \in \mathcal{X}^{j}} x \pi^{j}_{x} + s_{j}  c_{m+2^{p}, s_{j}} \right) = \bbP\left( \sum_{x \in \mathcal{X}^{j}} (n^{j}_{x}(s_{j}) - s_{j} \pi^{j}_{x}    ) \geq s_{j}  c_{m+2^{p}, s_{j}}/x \right) \\
&& \stackrel{(a)}{\leq}  \sum_{x \in \mathcal{X}^{j}}  \bbP\left( n^{j}_{x}(s_{j}) - s_{j} \pi^{j}_{x} \geq \frac{s_{j} c_{m+2^{p}}  }{x |\mathcal{X}^{j}|  } \right)  = \sum_{x \in \mathcal{X}^{j}} \bbP\left( \frac{\sum_{t=1}^{s_{j}} I\{X_{j}(t)=x\} - s_{j} \pi^{j}_{x}  }{ s_{j} {\hat{\pi}^{j}}_{x}  } \geq \frac{ c_{m+2^{p}, s_{j}} }{x |\mathcal{X}^{j}| \hat{\pi}^{j}_{x}} \right) \\
&& \stackrel{(b)}{\leq}  \sum_{x \in \mathcal{X}^{j}}  N_{\lambda^{j}} (m+2^{p})^{- \kappa \rho^{i}/28 x^{2}|\mathcal{X}^{j}|^{2}  (\hat{\pi}^{j}_{x})^{2}} ~~~  \stackrel{(c)}{\leq}   \frac{|\mathcal{X}|_{max}}{\pi_{min} } (m+2^{p})^{- \kappa \rho_{min}/28 |\mathcal{X}|_{max}^{2}}.
\end{eqnarray*}
The inequality (a) follows after some simple algebra, which we skip due to space limitations. The inequality (b) follows by defining the function $f(X_{j}(t))=(I\{X_{j}(t)=x\} - \pi^{j}_{x})/ {\hat{\pi}^{j}}_{x} $ and using the Lemma \ref{lemma_Le}. For inequality (c) we used the facts that $N_{\lambda^{j}} \leq 1/\pi_{min}$, $x_{max} \leq 1$ and ${\hat{\pi}_{max}} \leq 1 $. Thus,
\begin{equation}
\label{eq:event1prob_m}
\bbP(\overline{X}_{j}(m+2^{p})  \geq \mu_{j} + c_{m+2^{p}, s_{j}}) \leq D  (m+2^{p})^{- \kappa \rho_{min}/28 |\mathcal{X}|_{max}|^{2} }
\end{equation}
where $D=\frac{|\mathcal{X}|_{max}}{\pi_{min} } $. Similarly we can get,
\begin{equation}
\label{eq:event2prob_m}
\bbP(\overline{X}_{1}(m+2^{p}) \leq \mu_{1} - c_{m+2^{p}, s_{1}}) \leq D  (m+2^{p})^{- \kappa \rho_{min}/28 |\mathcal{X}|_{max}|^{2}  }
\end{equation}
For $l = \left \lceil 4 \kappa \log T/\Delta^{2}_{j}\right \rceil $, the last event in (\ref{events_singleuser}) is false. In fact, $\mu_{1} - \mu_{j} - 2 c_{m+2^{p}, s_{j}}$
\begin{eqnarray*}
=\mu_{1} - \mu_{j} - 2 \sqrt{\kappa \log(m+2^{p})/s_{j}} \geq \mu_{1} - \mu_{j} - \Delta_{j} = 0, ~ \text{for}~ s_{j} \geq  \left \lceil 4 \kappa \log T/\Delta^{2}_{j}\right \rceil.  ~\text{Thus},
\end{eqnarray*}
\begin{eqnarray}
\label{eq:njbd_M}
\bbE[n_{j}(T)] &\leq & \left \lceil \frac{ 4 \kappa \log T}{ \Delta_{j}^{2} } \right \rceil  + \sum_{m=1}^{\infty } \sum_{p=0}^{\infty} 2^{p} \sum_{s_{1}=1}^{m+2^{p}} \sum_{s_{j}=1}^{m+2^{p}} 2 D  (m+2^{p})^{-\frac{ \kappa \rho_{min}}{28 |\mathcal{X}|_{max}|^{2}   }}  \nonumber 
\end{eqnarray}
\begin{eqnarray}
&& = \left \lceil \frac{ 4 \kappa \log T}{ \Delta_{j}^{2} } \right \rceil  + 2 D \sum_{m=1}^{\infty } \sum_{p=0}^{\infty} 2^{p}  (m+2^{p})^{- \frac{\kappa \rho_{min} - 56|\mathcal{X}|_{max}^{2}  }{ 28|\mathcal{X}|_{max}^{2} }  }.
\end{eqnarray}
When $\kappa > 168 |\mathcal{X}|_{max}^{2}/\rho_{min} $, the above summation converges to a value less that $1$ and  we get
\begin{equation}
\label{eq:ucb4:nj:M}
 \bbE[n_{j}(T)] \leq \frac{ 4 \kappa \log T}{ \Delta^{2}_{j} } + (2D+1).
\end{equation}

Now, from the proof of  Theorem \ref{thm:ucb4_regret} (equation (\ref{eq:ucb4:m})), 
\begin{equation}
\label{eq:ucb4:m:M}
\bbE [m(T)] \leq \sum_{j>1}^{N} \bbE [n_{j}(T)] \cdot (1+\log T).
\end{equation}
Now, using inequality (\ref{eq:lemma_TeLi1}), the expected regret $\tilde{\sR}_{M, \tt UCB_4}(T)=$
\begin{eqnarray*}
 &=& \sum_{j>1}^{N} \bbE [n_j(T)] \cdot \Delta_{j} + C \bbE[m(T)]  + K_{\mathcal{X}, P} \leq  \Delta_{max} \sum_{j>1}^{N} \bbE [n_j(T)] + C \bbE[m(T)] + K_{\mathcal{X}, P}\\
&\leq & \left( \Delta_{max}  + C (1+\log T) \right) \sum_{j>1}^{N} \bbE [n_j(T)] + K_{\mathcal{X}, P}.
\end{eqnarray*}
by using  (\ref{eq:ucb4:m:M}). Now, by bound (\ref{eq:ucb4:nj:M}), we get the desired bound on the expected regret. \\
(ii) Replacing $\kappa$ with $\kappa_{t}$, equation (\ref{eq:njbd_M}) becomes 
\begin{eqnarray*}
\bbE[n_{j}(T)] &\leq & \left \lceil \frac{ 4 \kappa_{T} \log T}{ \Delta_{j}^{2} } \right \rceil  +  2 D \sum_{m=1}^{\infty } \sum_{p=0}^{\infty} 2^{p}  (m+2^{p})^{- \frac{\kappa_{m+2^{p}} \rho_{min} - 56|\mathcal{X}|_{max}^{2}  }{ 28|\mathcal{X}|_{max}^{2} }  } 
\end{eqnarray*}
Since, $\kappa_{t} \rightarrow \infty$ as $t \rightarrow \infty$, 
the exponent ${-\frac{\kappa_{m+2^{p}} \rho_{min} - 56|\mathcal{X}|_{max}^{2}  }{ 28|\mathcal{X}|_{max}^{2} }   }$  becomes smaller that $-4$ for sufficiently large $m$ and $p$, and the above summation converges, yielding the desired result. 
\end{proof}
We note that we have used the results in \cite{TeLi12} in the above proof. 
We note that the results for Markovian reward just presented extend easily even with finite precision indices. As before, suppose the cost of computation for $\e$-precision is $C(\e)$. We assume that  $C(\e) \rightarrow  \infty $ monotonically as $\epsilon \rightarrow 0$.  We formally state the following result, which we will use in section \ref{sec:dmab:M}.

\begin{theorem}\label{thm:ucb4_cmin_M}
(i) If $\Delta_{min}$,  $|\mathcal{X}|_{max}$ and $\rho_{min}$ are known, choose an $0 < \e < \D_{min}$, and a $\kappa > 168 |\mathcal{X}|_{max}^{2}/\rho_{min}$. Then, the expected regret using the ${\tt UCB_4}$ algorithm with the index defined as in (\ref{markov_single_index})  for the single player multi-armed bandit problem with Markovian rewards with $\e$-precise computations is given by
\begin{eqnarray*}
\tilde{\sR}_{M, \tt UCB_4}(T) &\leq&  \left(\Delta_{max}+C(\epsilon) (1+ \log T )\right) \cdot  \left(\sum_{j>1}^{N} \frac{4 \kappa \log T}{(\Delta_{j}-\epsilon)^{2} } + N (2D+1) \right). 
\end{eqnarray*}
where $D=\frac{|\mathcal{X}|_{max}}{\pi_{min} } $.  Thus, $\tilde{\sR}_{M, \tt UCB_4}(T) = O(\log^{2} T).$\\
(ii) If $\Delta_{min}$,  $|\mathcal{X}|_{max}$ and $\rho_{min}$ are unknown, choose a positive monotone sequences $\{\e_t\}$ such that and $\{\kappa_{t}\}$ such that $\kappa_{t} \leq t$, $\e_t \to 0$ and $\kappa_{t} \rightarrow \infty $ as $t \to \infty$.  Then, $\tilde{\sR}_{M, \tt UCB_4}(T) = O( C(\epsilon_{T}) \kappa_{T} \log^{2} T)$. We can choose $\{\epsilon_{t}\}$ and $\{\kappa_{t}\}$ as two arbitrarily slowly increasing sequences and thus the regret can be made arbitrarily close to $\log^{2}(T)$.  
\end{theorem}
The proof follows by a combination of the proof of the theorems \ref{thm:ucb4_cmin} and \ref{thm:ucb4_regret_M}, and is omitted.


\section{The Decentralized MAB problem with i.i.d. rewards}
\label{sec:dmab}

We now consider the decentralized multi-armed bandit problem with i.i.d. rewards wherein multiple players play at the same time.  
Players have no information about means or distribution of rewards from various arms. There are no \textit{dedicated control channels} for coordination or communication between the players. If two or more players pick the same arm, we assume that neither gets any reward. Tshis is an online learning problem of distributed bipartite matching. 

Distributed algorithms for bipartite matching algorithms are known \cite{Be88,ZaSpPa08} which determine an $\epsilon$-optimal matching with a `minimum' amount of information exchange and computation. However, every run of this distributed bipartite matching algorithm incurs a cost due to computation, and communication necessary  to exchange  some information for decentralization. Let $C$ be the cost per run, and $m(t)$ denote the number of times the distributed bipartite matching algorithm is run by time t. Then, under policy $\alpha$ the expected regret  is
\begin{equation}
\label{eq:dregret}  
\sR_{\alpha}(T)= T\sum_{i=1}^{M} \mu_{i,k_i^{**}}-\bbE_{\alpha}\left[\sum_{t=1}^{T} \sum_{i=1}^{M} X_{i,\alpha_i(t)}(t)\right] + C \bbE[m(T)].
\end{equation}
where $\mathbf{k}^{**}$ is the optimal matching as defined in equation (\ref{eq:optmatching}) in section \ref{subsec:model:iid}.\\
\noindent\textbf{Temporal Structure.} We divide time into frames. Each frame is one of two kinds: a \textit{decision frame}, and an \textit{exploitation frame}. In the decision frame, the index is recomputed, and the distributed bipartite matching algorithm run again to determine the new matching. The length of such a frame can be seen as cost of the algorithm. We further divide the decision frame into two phases, a \textit{negotiation phase} and an \textit{interrupt phase} (see Figure \ref{fig:negphasephy}). The information exchange needed to compute an $\e$-optimal matching is done in the \textit{negotiation phase}. In the \textit{interrupt phase}, a player signals to other players if his allocation has changed. In the exploitation frame, the current matching is \textit{exploited} without updating the indices. Later, we will allow the frame lengths  to increase with time.

We now present the ${\tt dUCB_4}$ algorithm, a decentralized version of ${\tt UCB_4}$. For each player $i$ and each arm $j$, we define a ${\tt dUCB_4}$ index at the end of frame $t$ as
\begin{equation}
\label{eq:index}
g_{i,j}(t) := \overline{X}_{i,j}(t)+\sqrt{\frac{(M+2)\log n_i(t)} {n_{i,j}(t)}},
\end{equation}
where $n_i(t)$ is the number of successful plays (without collisions) of player $i$ by frame $t$, $n_{i,j}(t)$ is the number of times player $i$ picks  arm $j$ successfully by frame $t$. $\overline{X}_{i,j}(t)$ is the sample mean of rewards from arm $j$ for player $i$ from $n_{i,j}(t)$ samples.  Let $g(t)$ denote the vector $(g_{i,j}(t), 1 \leq i \leq M, 1 \leq j \leq N)$. Note that $g$ is computed only in the decision frames using the information available upto that time.  Each player now uses the ${\tt dUCB_4}$ algorithm. We will refer to an $\epsilon$-optimal distributed bipartite matching algorithm as ${\tt dBM_{\e}}(g(t))$ that yields a solution $\mathbf{k}^*(t):=(k_1^*(t),\ldots, k_M^*(t)) \in \sP(N) $ such that $\sum_{i=1}^{M} g_{i, k^{*}_{i}(t)}(t) \geq  \sum_{i=1}^{M}   g_{i, k_{i}}(t) ) - \epsilon, ~ \forall \mathbf{k} \in \sP(N)  , \mathbf{k} \neq \mathbf{k}^*$. Let $\mathbf{k}^{**} \in \sP(N)$ be such that $\mathbf{k}^{**} \in \arg \max_{\mathbf{k} \in \sP(N) } \sum_{i=1}^{M} \mu_{i,\mathbf{k}_{i}},$ i.e., an optimal bipartite matching with expected rewards from each matching.
Denote $\mu^{**} := \sum_{i=1}^{M} \mu_{i, \mathbf{k}_i^{**}}$, and define $\Delta_{\mathbf{k}} :=  \mu^{**} - \sum_{i=1}^{M} \mu_{i,\mathbf{k}_{i} },~ \mathbf{k} \in \sP(N)$.
Let $\Delta_{min}  = \min_{\mathbf{k} \in \sP(N), \mathbf{k} \neq \mathbf{k}^{**}} \Delta_{\mathbf{k}}$ and $\Delta_{max}  = \max_{\mathbf{k} \in \sP(N)} \Delta_{\mathbf{k}}$. 
We assume that $\Delta_{min}>0$. 


\begin{algorithm}
\caption{$\tt dUCB_4$ for User $i$}
\label{algo:dUCB_4}
{\small
\begin{algorithmic}[1]
\STATE {\bf Initialization:} Play a set of matchings so that each player plays each arm at least once. Set counter $\eta = 1$.
\WHILE {($t \leq T$)}
\IF {($\eta = 2^{p}~ \text{for some}~ p = 0, 1, 2, \cdots$)}
\STATE //\textit{Decision frame:}
\STATE Update $g(t)$;
\STATE Participate in the ${\tt dBM_{\e}}(g(t))$  algorithm to obtain a match $k_i^*(t)$;
\IF{$(k_i^*(t) \neq k_i^*(t-1))$}
\STATE Use \textit{interrupt phase} to signal an INTERRUPT to all other players about changed allocation;
\STATE Reset $\eta=1$;
\ENDIF
\IF{(Received an INTERRUPT)}
\STATE Reset $\eta=1$;
\ENDIF
\ELSE 
\STATE // \textit{Exploitation frame:}
\STATE $k_i^*(t) = k_i^*(t-1)$;
\ENDIF
\STATE Play arm $k_i^*(t)$;
\STATE Increment counter $\eta=\eta+1$, $t=t+1$;
\ENDWHILE
\end{algorithmic}
}
\end{algorithm} 

In the $\tt dUCB_4$ algorithm, at the end of every decision frame, the ${\tt dBM_{\e}}(g(t))$ will give a legitimate matching with no two players colliding on any arm. Thus, the regret accrues either if the matching $\mathbf{k}(t)$ is not the optimal matching $\mathbf{k}^{**}$, or if a decision frame is employed by the players to recompute the matching. Every time a frame is a decision frame, it adds a cost $C$ to the regret. The cost $C$ depends on two parameters: (a) the precision of the bipartite matching algorithm $\epsilon_{1}>0$, and (b) the precision of the index representation $\epsilon_{2}>0$. A bipartite matching algorithm has an $\epsilon_{1}$-precision if it gives an $\epsilon_{1}$-optimal matching. This would happen, for example, when such an algorithm is run only for a finite number of rounds. The index  has an $\epsilon_{2}$-precision if any two indices are not distinguishable if they are closer than $\epsilon_{2}$. This can happen for example when indices must be communicated to other players in a finite number of bits.

Thus, the cost $C$ is a function of $\epsilon_{1}$ and $\epsilon_{2}$, and can be denoted as $C(\epsilon_{1}, \epsilon_{2})$, with $C(\epsilon_{1}, \epsilon_{2}) \rightarrow \infty$ as $\epsilon_{1}$ or  $\epsilon_{2} \rightarrow 0$. Since, $\epsilon_{1}$ and $\epsilon_{2}$ are the parameters that are fixed \textit{a priori}, we consider $\epsilon=\min(  \epsilon_{1},\epsilon_{2})$ to specify  both precisions. We denote the cost as $C(\epsilon)$. 

We first show that if $\Delta_{min}$ is known, we can choose an $\epsilon < \Delta_{min}/(M+1)$, so that ${\tt dUCB_4}$ algorithm will achieve order log-squared regret growth with $T$. If $\Delta_{min}$ is not known, we can pick a positive monotone sequence $\{\epsilon_{t}\}$  such that $\e_t \to 0$, as $t \to \infty$. In a decentralized bipartite matching algorithm, the precision $\e$ will depend on the amount of information exchanged in the decision frames. It, thus, is some monotonically decreasing function $\e=f(L)$ of their length $L$ such that $\e \to 0$ as $L \to \infty$. Thus, we must pick a positive monotone sequence $\{L_t\}$ such that $L_t \to \infty$. Clearly, $C(f(L_{t})) \to \infty$ as $t \to \infty$.  This can happen arbitrarily slowly. 

\begin{theorem}\label{thm:ducb4}
(i) Let $\epsilon > 0$ be the precision of the bipartite matching algorithm and the precision of the index representation. If $\Delta_{min}$ is known, choose $\epsilon > 0$ such that $\epsilon < \D_{min}/(M+1)$. Let $L$ be the length of a frame. Then, the expected regret of the ${\tt dUCB_4}$ algorithm is 
\begin{eqnarray*}
 \tilde{\sR}_{\tt dUCB_4}(T) &\leq & (L\Delta_{max} + C(f(L))   (1+\log T) ) \cdot \left( \frac{4 M^{3} (M+2) N \log T }{  (\Delta_{min} - ( (M+1) \epsilon)^{2} }   + NM (2M+1)  \right).
\end{eqnarray*}
Thus, \(\tilde{\sR}_{\tt dUCB_4}(T)= O(\log^{2} T).\)\\
(ii) When $\Delta_{min}$ is  unknown, denote $\epsilon_{min} = \Delta_{min}/(2(M+1))$ and let $L_{t} \rightarrow \infty$ as $t \rightarrow \infty$. Then, there exists a $t_0>0$ such that for all $T > t_0$,
\begin{eqnarray*}
 \tilde{\sR}_{\tt dUCB_4}(T) && \leq  (L_{t_{0}}\Delta_{max}+C(f(L_{t_{0}})) t_{0}  + (L_{T}\Delta_{max} + C(f(L_{T})) (1+\log T) ) \cdot \\
 &&  \hspace{2cm} \left( \frac{4 M^{3} (M+2) N \log T }{  (\Delta_{min} - \epsilon_{min})^{2} }   + NM (2M+1)  \right),
\end{eqnarray*}
where $t_{0}$ is the smallest $t$ such that $f(L_{t_{0}}) < \epsilon_{min}$. Thus by choosing an arbitrarily slowly increasing sequence $\{L_{t}\}$ we can make the regret arbitrarily close to $\log^{2} T$.
%
\end{theorem}
\begin{proof}
(i) First, we obtain a bound for $L=1$. Then, appeal to a result like Theorem \ref{thm:ucb1} to obtain the result for general $L$. The implicit dependence between $\epsilon$ and $L$ through the function $f(\cdot)$ does not affect this part of the analysis. Details are omitted due to space limitations. 

We first upper bound the number of sub-optimal plays. We define $\tilde{n}_{i,j}(t), 1 \leq i \leq M, 1 \leq j \leq N$ as follows: Whenever the ${\tt dBM_{\e}}(g(t))$ algorithm gives a non-optimal matching $\mathbf{k}(t)$,  $\tilde{n}_{i,j}(t)$ is increased by one for some $(i,j) \in \arg \min_{1 \leq i \leq M, 1 \leq j \leq N  } n_{i,j}(t)$. Let $\tilde{n}(T)$ denote the total number of suboptimal plays. Then, clearly, $\tilde{n}(T) = \sum_{i=1}^{M} \sum_{j=1}^{N} \tilde{n}_{i,j}(T)$. So, in order to get a bound on $\tilde{n}(T)$ we first get a bound on $\tilde{n}_{i,j}(T)$. 

Let $\tilde{I}_{i,j}(t)$ be the indicator function which is equal to $1$ if $\tilde{n}_{i,j}(t)$ is incremented by one, at time $t$. When  $\tilde{I}_{i,j}(t)=1$, there will be a corresponding  matching $\mathbf{k}(t) \neq \mathbf{k}^{**}$ such that $k_{i}(t)=j$. In the following, we denote it as $\mathbf{k}$, omitting the time index. A non-optimal matching $\mathbf{k}$ is selected if the event $\bigg\{  \sum_{i=1}^{M}  g_{i,k^{**}_{i} }(m+2^{p}-1) \leq (M+1) \e + \sum_{i=1}^{M}  g_{i,k_{i}}(m+2^{p}-1)\bigg\}$ happens. If each index has an error of at most $\epsilon$, the sum of $M$ terms may introduce an error of atmost $M\epsilon$. In addition, the distributed bipartite matching algorithm ${\tt dBM_{\e}}$ itself yields only an $\epsilon$-optimal matching. This accounts for the term $(M+1)\e$ above. Since the initial steps are similar to that in Theorem \ref{thm:ucb4_regret}, we skip those steps. Thus, similar to the equation (\ref{eq:singleusr_njbd_st1}), we get $\tilde{n}_{i,j}(T) \leq$
\begin{eqnarray}
\label{eq:njbd_st1}
 l &+& \sum_{m=1}^{T} \sum_{p=0}^{\infty } 2^{p} I \bigg\{  \sum_{i=1}^{M}  g_{i,k^{**}_{i} }(m+2^{p}-1) \leq (M+1) \e + \sum_{i=1}^{M}  g_{i,k_{i}}(m+2^{p}-1),  \tilde{n}_{i,j}(m-1) \geq l \bigg\} \nonumber \\
&\leq & l + \sum_{m=1}^{T} \sum_{p=0}^{\infty } 2^{p} I \bigg \{  \sum_{i=1}^{M}  \bigg( \overline{X}_{i,k^{**}_{i}}(m+2^{p}-1)   + c_{m+2^{p}-1, n_{i,k^{**}_{i}}(m+2^{p}-1)} \bigg)   \nonumber \\
  ~~&& \hspace{2cm}  \leq  (M+1)\e + \sum_{i=1}^{M}  \overline{X}_{i,k_{i}}(m+2^{p}-1) +  c_{m+2^{p}-1, n_{i,k_{i}}(m+2^{p}-1)} , \tilde{n}_{i,j}(m-1) \geq l \bigg \}  \nonumber \\
&\leq & l + \sum_{m=1}^{T} \sum_{p=0}^{\infty } 2^{p} I \bigg\{ \min_{1 \leq s_{1,k^{**}_{1}}, \ldots, s_{M,k^{**}_{M}} < m+2^{p}} \sum_{i=1}^{M} \left( \overline{X}_{i,k^{**}_{i}}(m+2^{p}-1)+c_{m+2^{p}-1, s_{i,k^{**}_{i}}} \right) \nonumber \\
&& \hspace{2cm} \leq (M+1)\e + \max_{l \leq s_{1,k_{1}}^{'}, \ldots, s_{M,k_{M}}^{'} < m+2^{p}} \sum_{i=1}^{M}  \left( \overline{X}_{i,k_{i}}(m+2^{p}-1)+c_{m+2^{p}-1, s_{i,k_{i}}^{'}} \right) \bigg \} \nonumber \\
&& \leq  l + \sum_{m=1}^{\infty} \sum_{p=0}^{\infty } 2^{p} \sum_{s_{1,k^{**}_{1}}=1}^{m+2^{p}} \ldots \sum_{s_{M,k^{**}_{M}}=1}^{m+2^{p}} \sum_{s^{'}_{1,k_{1}}=1}^{m+2^{p}} \ldots \sum_{s^{'}_{M,k_{M}}=1}^{m+2^{p}} I\bigg \{   \sum_{i=1}^{M} \left( \overline{X}_{i,k^{**}_{i}}(m+2^{p})+c_{m+2^{p}, s_{i,k^{**}_{i}}} \right) \nonumber \\
&& \hspace{4cm} \leq (M+1)\e + \sum_{i=1}^{M} \left( \overline{X}_{i,k_{i}}(m+2^{p})+c_{m+2^{p}, s_{i,k_{i}}^{'}} \right) \bigg \}.
\end{eqnarray}
Now, it is easy to observe that the event
\begin{eqnarray*}
 \bigg \{  \sum_{i=1}^{M} \left( \overline{X}_{i,k^{**}_{i}}(m+2^{p})+c_{m+2^{p}, s_{i,k^{**}_{i}}} \right) \nonumber \leq (M+1)\e +  \sum_{i=1}^{M} \left( \overline{X}_{i,k_{i}}(m+2^{p})+c_{m+2^{p}, s_{i,k_{i}}^{'}} \right) \bigg \}
\end{eqnarray*}
implies at least one of the following events: 
\begin{eqnarray}
\label{events_multiuser} 
A_{i} &:=& \bigg \{\overline{X}_{i,k^{**}_{i}}(m+2^{p}) \leq  \mu_{i,k^{**}_{i}} - c_{m+2^{p}, s_{i,k^{**}_{i}}} \bigg \},  \nonumber \\
B_{i} &:=& \bigg \{ \overline{X}_{i,k_{i}}(m+2^{p}) \geq \mu_{i,k_{i}} +  c_{m+2^{p}, s_{i,k_{i}}^{'}} \bigg \},  1 \leq i \leq M,  \nonumber \\
C &:= &\bigg \{ \sum_{i=1}^{M} \mu_{i,k^{**}_{i}}  < (M+1) \e + \sum_{i=1}^{M} \mu_{i,k_{i}}  + 2 \sum_{i=1}^{M} c_{m+2^{p}, s_{i,k_{i}}^{'}} \bigg \} \\
D &:=&\bigg \{ (M+1)\e  >    \sum_{i=1}^{M} \mu_{i,k^{**}_{i}}  - \sum_{i=1}^{M} \mu_{i,k_{i}} \bigg \}. \nonumber
\end{eqnarray}
Using the Chernoff-Hoeffding inequality, we get 
\(
\bbP(A_{i}) \leq  (m+2^{p})^{-2(M+2)}, ~~\bbP(B_{i}) \leq  (m+2^{p})^{-2(M+2)}, ~ 1 \leq i \leq M. 
\)
For $ l \geq \left \lceil \frac{4 M^{2} (M+2) \log T }{  (\Delta_{min} -(M+1)\e )^{2} } \right \rceil $, we get \\ $ \sum_{i=1}^{M} \mu_{i,k^{**}_{i}}  - \sum_{i=1}^{M} \mu_{i,k_{i}} - (M+1)\epsilon - 2 \sum_{i=1}^{M} c_{m+2^{p}, s_{i,k_{i}}^{'}} $
\begin{eqnarray}
\label{events_CD}
&\geq & \sum_{i=1}^{M} \mu_{i,k^{**}_{i}}  - \sum_{i=1}^{M} \mu_{i,k_{i}} - (M+1)\epsilon  - 2 M \sqrt{\frac{(M+2) \log (m+2^{p}) }{l}  } \nonumber  \\
& \geq & \sum_{i=1}^{M} \mu_{i,k^{**}_{i}}  - \sum_{i=1}^{M} \mu_{i,k_{i}} - (M+1)\epsilon  - (\Delta_{min} - (M+1)\epsilon) \geq 0
\end{eqnarray}
The event $D$ is false by assumption. So, we get, $\bbE[\tilde{n}_{i,j}(T)]$ 
\begin{eqnarray}
\label{eq:ntilde}
&& \leq \left \lceil \frac{4 M^{2} (M+2) \log T }{  (\Delta_{min} -(M+1)\e )^{2} } \right \rceil + \sum_{m=1}^{\infty} \sum_{p=0}^{\infty } 2^{p} \sum_{s_{1,k^{**}_{1}}=1}^{m+2^{p}} \ldots \sum_{s_{M,k^{**}_{M}}=1}^{m+2^{p}} \sum_{s^{'}_{1,k_{1}}=1}^{m+2^{p}} \ldots \sum_{s^{'}_{M,k_{M}}=1}^{m+2^{p}} 2M (m+2^{p})^{-2(M+2)} \nonumber \\
&& \leq  \left \lceil \frac{4 M^{2} (M+2) \log T }{  (\Delta_{min} - (M+1)\e)^{2} } \right \rceil   + 2 M \sum_{m=1}^{\infty}  \sum_{p=0}^{\infty } 2^{p} (m+2^{p})^{-4} \nonumber \\
&& \leq  \frac{4 M^{2} (M+2)  \log T }{  (\Delta_{min} - (M+1)\e)^{2} }   + (2M+1).
\end{eqnarray}
Now, putting it all together, we get
\begin{eqnarray}
\label{dUCB_subopt}
\bbE[\tilde{n}(T)] &=& \sum_{i=1}^{M} \sum_{j=1}^{N}  \bbE[\tilde{n}_{i,j}(T)] \leq  \frac{4 M^{3} (M+2) N  \log T }{  (\Delta_{min} - (M+1)\e)^{2} }   + (2M+1) MN. \nonumber
\end{eqnarray}
Now, by the proof of Theorem \ref{thm:ucb4_regret} (c.f. equation(\ref{eq:ucb4:m}), $\bbE [m(T)]  \leq \bbE[\tilde{n}(T)] (1+ \log T).$ 
We can now bound the regret,
$\tilde{\sR}_{\tt dUCB_4}(T) = \sum_{k \in \sP(N), k \neq k^{**}} \Delta_{k} \sum_{i=1}^{M}  \bbE[\tilde{n}_{i,k_{i}}(T)] + C \bbE[m(T)] $
\begin{eqnarray*}
& \leq &  \Delta_{max} \sum_{k \in \sP(N), k \neq k^{**}} \sum_{i=1}^{M}  \bbE[\tilde{n}_{i,k_{i}}(T)] + C \bbE[m(T)]\\
& =& \Delta_{max} \bbE[\tilde{n}(T)] + C \bbE[m(t)].
\end{eqnarray*}
For a general $L$, by Theorem \ref{thm:ucb1} we get 
\begin{eqnarray*}
\tilde{\sR}_{\tt dUCB_4}(T)  \leq  L \Delta_{max} \bbE[\tilde{n}(T)] + C(f(L)) \bbE[m(T)]  \leq  (L \Delta_{max} + C(f(L))  (1+ \log T)) \bbE[\tilde{n}(T)].
\end{eqnarray*}
 Now, using the bound (\ref{dUCB_subopt}), we get the desired upper bound on the expected regret. 
\\
(ii) Since $\e_t=f(L_t)$ is a monotonically decreasing function of $L_t$ such that $\e_t \to 0$ as $L_t \to \infty$, there exists a $t_{0}$ such that  for $t > t_{0}$, $\e_t < \e_{min}$.  We may get a linear regret upto time $t_0$ but after that by the analysis of Theorem \ref{thm:ucb4_regret},  regret grows only sub-linearly. Since $C(\cdot)$ is monotonically increasing, $C(f(L_T)) \geq C(f(L_t)), \forall t \leq T$,  we get the desired result. The last part is illustrative and can be trivially established using the obtained bound on the regret in (ii).
\end{proof}

\noindent \textbf{Remarks.} 1. We note that in the initial steps, our proof followed the proof of the main result in \cite{GaKrJa12}. \\
2. The ${\tt UCB_2}$ algorithm described in \cite{AuCeFi02} performs computations only at exponentially spaced time epochs. So, it is natural to imagine that a decentralized algorithm based on it could be developed, and get a better regret bound. Unfortunately, the single player ${\tt UCB_2}$ algorithm has an obvious weakness: regret is linear in the number of arms. Thus, the decentralized/combinatorial extension of ${\tt UCB_2}$ would yield regret growing exponentially in the number of players and arms. We use a similar index but a different scheme, allowing us to achieve \textit{poly-log} regret growth and a linear memory requirement for each player.

\section{The Decentralized MAB problem with Markovian rewards}
\label{sec:dmab:M}
Now, we consider the decentralized MAB problem with $M$ players and $N$ arms where the rewards obtained each time when an arm is pulled are not i.i.d. but come from a Markov chain. The reward that player $i$ gets from arm $j$ (when there is no collision) $X_{ij}$, is modelled as an irreducible, aperiodic, reversible Markov chain on a finite state space $\mathcal{X}^{i,j}$ and represented by a transition probability matrix $P^{i,j}:=\left(p^{i,j}_{x,x^{'}} : x, x^{'} \in \mathcal{X}^{i,j}\right)$. Assume that $\mathcal{X}^{i,j} \in (0, 1]$. Let  $X_{i,j}(1), X_{i,j}(2), \ldots$ denote the successive rewards from arm $j$ for player $i$.  All arms are mutually independent for all players.  Let $\mathbf{\pi}^{i,j}:=\left(\pi^{i,j}_{x}, x \in \mathcal{X}^{i,j}\right)$ be the stationary distribution of the Markov chain $P^{i,j}$. The mean reward from arm $j$ for player $i$ is defined as $\mu_{i,j}:=\sum_{x \in \mathcal{X}  ^{i,j} } x \pi^{i,j}_{x} $. Note that the Markov chain represented by $P^{i,j}$ makes a state transition only when player $i$ plays arm $j$. Otherwise, it remains \emph{rested}.  As described in the previous section,  $n_i(t)$ is the number of successful plays (without collisions) of player $i$ by frame $t$, $n_{i,j}(t)$ is the number of times player $i$ picks  arm $j$ successfully by frame $t$ and $\overline{X}_{i,j}(t)$ is the sample mean of rewards from arm $j$ for player $i$ from $n_{i,j}(t)$ samples. Denote $\Delta_{min}:= \min_{\mathbf{k} \in \sP(N), \mathbf{k} \neq \mathbf{k}^{**}} \Delta_{\mathbf{k}}$ and $\Delta_{max}:= \max_{\mathbf{k} \in \sP(N)} \Delta_{\mathbf{k}}$. Denote $\pi_{min}:=\min_{1 \leq i \leq M, 1 \leq j \leq N, x \in \mathcal{X}^{i,j}} \pi^{i,j}_{x} $, $x_{max}:=\max_{1 \leq i \leq M, 1 \leq j \leq N, x \in \mathcal{X}^{i,j}} x$ and $x_{min}:=\min_{1 \leq i \leq M, 1 \leq j \leq N, x \in \mathcal{X}^{i,j}}x$. Let ${\hat{\pi}^{i,j}}_{x}:=\max\{\pi^{i,j}_{x}, 1-\pi^{i,j}_{x} \}$ and $\hat{\pi}_{max}:= \max_{1 \leq i \leq M, 1 \leq j \leq N, x \in \mathcal{X}^{i,j}} {\hat{\pi}^{i,j}}_{x} $. Let  $|\mathcal{X}^{i,j}|$ denote the cardinality of the state space $\mathcal{X}^{i,j}$, $|\mathcal{X}|_{max}:=\max_{1 \leq i \leq M, 1 \leq j \leq N } |\mathcal{X}^{i,j}|$. Let $\rho^{i,j}$ be the  eigenvalue gap, $1 - \lambda_{2}$, where $\lambda_{2}$ is the second largest eigenvalue of the matrix ${P^{i,j}}^{2}$. Denote $\rho_{max}:= \max_{1 \leq i \leq M, 1 \leq j \leq N } \rho^{i,j}$ and $\rho_{min}:=\min_{1 \leq i \leq M, 1 \leq j \leq N }\rho^{i,j}$.

The total reward obtained by time $T$ is $S_{T}  =  \sum_{j=1}^{N} \sum_{i=1}^{M} \sum_{s=1}^{n_{i,j}(T)} X_{i,j}(s)$ and the regret is
\begin{equation}
\label{eq:dregret:M}  
\tilde{\sR}_{M, \alpha}(T) := T\sum_{i=1}^{M} \mu_{i,k_i^{**}}-\bbE_{\alpha}\left[\sum_{j=1}^{N} \sum_{i=1}^{M} \sum_{s=1}^{n_{i,j}(T)} X_{i,j}(s)\right] + C \bbE[m(T)].
\end{equation}
Define the index
 \begin{equation}
\label{eq:dindex:M}
g_{i,j}(t) := \overline{X}_{i,j}(t)+\sqrt{\frac{\kappa  \log n_i(t)} {n_{i,j}(t)}} 
\end{equation} where $\kappa$ be any constant such that  $\kappa > (112+56M) |\mathcal{X}|_{max}^{2}/\rho_{min}$.

We need the following lemma to prove the regret bound. 
\begin{lemma}
If the reward of each player-arm pair $(i,j)$ is given by a Markov chain, satisfying the properties of Lemma \ref{lemma_AnVa1}, then  under any policy $\alpha$ 
\begin{equation}
\tilde{\sR}_{M, \tt \alpha }(T) \leq \sum_{k \in \sP(N), k \neq k^{**}} \Delta_{k}   \bbE[n^{k}(T)] + C \bbE[m(T)] + \tilde{K}_{\mathcal{X}, P}
\end{equation}
where $n^{k}(T)$ is the number of times that the matching $k$ occurred by the time $T$ and  $\tilde{K}_{\mathcal{X}, P}$ is defined as $\tilde{K}_{\mathcal{X}, P} = \sum_{j=1}^{N} \sum_{i=1}^{M} \sum_{x \in \mathcal{X}^{i,j}} x/\pi^{j}_{min} $
\end{lemma}
\begin{proof}
Let  $(X_{i,j}(1), X_{i,j}(2), \ldots)$ denote the successive rewards for player $i$ from arm $j$.  Let $\mathcal{F}^{i,j}_{t}$ denote the $\sigma$-algebra generated by $(X_{i,j}(1), \ldots, X_{i,j}(t))$, $\mathcal{F}^{i,j}=\vee_{t \geq 1}\mathcal{F}^{i,j}_{t}$ and $\mathcal{G}^{i,j}=\vee_{(k,l) \neq (i,j) } F^{k,l}$.  Since arms are independent, $\mathcal{G}^{i,j}$ is independent of $\mathcal{F}^{i,j}$. Clearly, $n_{i,j}(T)$ is a stopping time with respect to $\mathcal{F}^{i,j} \vee \mathcal{G}^{i,j}_{T}$. The total reward is 
\[ S_{T}  =  \sum_{j=1}^{N} \sum_{i=1}^{M} \sum_{t=1}^{n_{i,j}(T)} X_{i,j}(t) =  \sum_{j=1}^{N} \sum_{i=1}^{M} \sum_{x \in \mathcal{X}^{i,j}} x N(x, n_{i,j}(T) )  \]
where $N(x, n_{i,j}(T)) := \sum_{t=1}^{n_{i,j}(T)} I\{X_{i,j}(t)=x\} $. Taking expectations and using the Lemma \ref{lemma_AnVa1},
\begin{equation*}
 \left \lvert \bbE[S_{T}] - \sum_{j=1}^{N} \sum_{i=1}^{M}  \sum_{x \in \mathcal{X}^{i,j}} x \pi^{i,j}_{x} \bbE[n_{i,j}(T)] \right \rvert  \leq \sum_{j=1}^{N} \sum_{i=1}^{M}  \sum_{x \in \mathcal{X}^{i,j}} x/\pi^{i,j}_{min}
\end{equation*}
which implies, 
\begin{equation*}
 \left \lvert \bbE[S_{T}] - \sum_{j=1}^{N} \sum_{i=1}^{M} \mu_{i,j} \bbE[n_{i,j}(T)] \right \rvert  \leq \tilde{K}_{\mathcal{X}, P}
\end{equation*}
where $\tilde{K}_{\mathcal{X}, P} = \sum_{j=1}^{N} \sum_{i=1}^{M} \sum_{x \in \mathcal{X}^{i,j}} x/\pi^{i,j}_{min}$.
Now, 
\begin{eqnarray*}
\sum_{j=1}^{N} \sum_{i=1}^{M} \mu_{i,j} \bbE[n_{i,j}(T)]  && = \sum_{j=1}^{N} \sum_{i=1}^{M} \sum_{ k \in \sP(N), (i,j) \in k } \mu_{i,k_{i}} \bbE[n_{i,k_{i}}(T)] = \sum_{ k \in \sP(N)} \sum_{i=1}^{M} \mu_{i,k_{i}} \bbE[n_{i,k_{i}}(T)] \\
&& = \sum_{ k \in \sP(N)}  \mu^{k} \bbE[n^{k}(T)] 
\end{eqnarray*}
where $\mu^{k} = \sum_{i=1}^{M} \mu_{i, k_{i}}$. Since regret is defined as in the equation (\ref{eq:dregret:M}), 
\begin{equation}
\left \lvert \tilde{\sR}_{M, \alpha}(T) - \left( T \mu^{**} - \sum_{ k \in \sP(N), (i,j) \in k } \mu_{i,k_{i}} \bbE[n_{i,k_{i}}(T)]  + C \bbE_{\alpha}[m(T)] \right) \right \rvert \leq \tilde{K}_{\mathcal{X},P}.
\end{equation}
\end{proof}
The main result of this section is the following.
\begin{theorem}
\label{thm:ducb4_regret_M} 
(i) Let $\epsilon > 0$ be the precision of the bipartite matching algorithm and the precision of the index representation. If $\Delta_{min}$,  $|\mathcal{X}|_{max}$ and $\rho_{min}$ are known, choose $\epsilon > 0$ such that $\epsilon < \D_{min}/(M+1)$ and  $\kappa > (112+56M) |\mathcal{X}|_{max}^{2}/\rho_{min}$. Let $L$ be the length of a frame. Then, the expected regret of the ${\tt dUCB_4}$ algorithm with index (\ref{eq:dindex:M}) for the decentralized MAB problem with Markovian rewards and per computation cost $C$  is given by \\
$\tilde{\sR}_{M, \tt dUCB_4}(T)$
\begin{eqnarray*}
 &\leq & (L\Delta_{max} + C(f(L))   (1+\log T) ) \cdot \left( \frac{4 M^{3} \kappa N  \log T }{  (\Delta_{min} - (M+1)\e)^{2} }   + (2MD+1) MN  \right) + \tilde{K}_{\mathcal{X}, P}. 
\end{eqnarray*}
Thus, \(\tilde{\sR}_{M, \tt dUCB_4}(T)= O(\log^{2} T).\)\\
(ii)  If $\Delta_{min}$,  $|\mathcal{X}|_{max}$ and $\rho_{min}$ are unknown, denote $\epsilon_{min} = \Delta_{min}/(2(M+1))$ and let $L_{t} \rightarrow \infty$ as $t \rightarrow \infty$. Also,  choose a positive monotone sequence $\{\kappa_{t}\}$ such that $\kappa_{t} \rightarrow \infty $ as $t \rightarrow \infty$ and $\kappa_{t} \leq t$.  Then, $\tilde{\sR}_{M, \tt dUCB_4}(T) = O( C(f(L_{T})) \kappa_{T} \log^{2} T)$. Thus by choosing an arbitrarily-slowly increasing sequences, we can make the regret arbitrarily close to $\log^{2} T$.
\end{theorem}
\begin{proof}
(i) We skip the initial steps as they are  same as in the  proof of  Theorem \ref{thm:ducb4}. We start by bounding $\tilde{n}_{i,j}(T)$ as defined in the  proof of  Theorem \ref{thm:ducb4}. Then, from equation (\ref{eq:njbd_st1}), we get
$\tilde{n}_{i,j}(T) $
\begin{eqnarray} 
&& \leq  l + \sum_{m=1}^{\infty} \sum_{p=0}^{\infty } 2^{p} \sum_{s_{1,k^{**}_{1}}=1}^{m+2^{p}} \ldots \sum_{s_{M,k^{**}_{M}}=1}^{m+2^{p}} \sum_{s^{'}_{1,k_{1}}=1}^{m+2^{p}} \ldots \sum_{s^{'}_{M,k_{M}}=1}^{m+2^{p}} I\{  \sum_{i=1}^{M} \left( \overline{X}_{i,k^{**}_{i}}(m+2^{p})+c_{m+2^{p}, s_{i,k^{**}_{i}}} \right) \nonumber \\
&& \hspace{4cm} \leq  (M+1) \epsilon + \sum_{i=1}^{M} \left( \overline{X}_{i,k_{i}}(m+2^{p})+c_{m+2^{p}, s_{i,k_{i}}^{'}} \right) \} 
\end{eqnarray}
Now, the event in the parenthesis $\{\cdot \}$ above implies at least one of the events ($A_{i}, B_{i}, C, D$) given  in the display (\ref{events_multiuser}). From the proof of Theorem \ref{thm:ucb4_regret_M} (equations (\ref{eq:event1prob_m}, \ref{eq:event2prob_m}),
\(
\bbP(A_{i}) \leq  D  (m+2^{p})^{- \frac{\kappa \rho_{min}}{28 |\mathcal{X}|_{max}|^{2}} }, \hspace{1cm} \bbP(B_{i}) \leq  D  (m+2^{p})^{- \frac{\kappa \rho_{min}}{28 |\mathcal{X}|_{max}|^{2}} }, 1 \leq i \leq M.
\)
Similar to the steps in display (\ref{events_CD}), we can show that the event $C$ is false. Also, the event $D$ is false by assumption. So, similar to the proof of the Theorem \ref{thm:ducb4} (c.f. display (\ref{eq:ntilde})  we get,
\begin{eqnarray*}
\bbE[\tilde{n}_{i,j}(T)] &\leq & \left \lceil \frac{4 M^{2} \kappa \log T }{  (\Delta_{min} -(M+1)\e )^{2} } \right \rceil + \sum_{m=1}^{\infty} \sum_{p=0}^{\infty } 2^{p} \sum_{s_{1,k^{**}_{1}}=1}^{m+2^{p}} \ldots \sum_{s_{M,k^{**}_{M}}=1}^{m+2^{p}} \\
~~&& \hspace{3cm} \sum_{s^{'}_{1,k_{1}}=1}^{m+2^{p}} \ldots \sum_{s^{'}_{M,k_{M}}=1}^{m+2^{p}} 2M D (m+2^{p})^{- \frac{\kappa \rho_{min}}{28 |\mathcal{X}|_{max}|^{2}} }  \\
& \leq & \left \lceil \frac{4 M^{2} \kappa \log T }{  (\Delta_{min} - (M+1)\e)^{2} } \right \rceil + 2 M D \sum_{m=1}^{\infty}  \sum_{p=0}^{\infty } 2^{p} (m+2^{p})^{- \frac{\kappa \rho_{min} - 56M|\mathcal{X}|_{max}^{2} }{ 28|\mathcal{X}|_{max}^{2}  }  }  \\
&\leq &  \frac{4 M^{2} \kappa  \log T }{  (\Delta_{min} - (M+1)\e)^{2} }   + (2MD+1).
\end{eqnarray*}
when $\kappa > (112+56M) |\mathcal{X}|_{max}^{2}/\rho_{min}$. Now, putting it all together, we get
\begin{eqnarray}
\label{dUCB_subopt_M}
\bbE[\tilde{n}(T)] &= & \sum_{i=1}^{M} \sum_{j=1}^{N}  \bbE[\tilde{n}_{i,j}(T)] \nonumber \leq   \frac{4 M^{3} \kappa N  \log T }{  (\Delta_{min} - (M+1)\e)^{2} }   + (2MD+1) MN. 
\end{eqnarray}
Now, by proof of the Theorem  \ref{thm:ucb4_regret} (equation (\ref{eq:ucb4:m})), $\bbE [m(T)]  \leq \bbE[\tilde{n}(T)] (1+ \log T).$
We can now bound the regret,
\begin{eqnarray*}
 \tilde{\sR}_{M, \tt dUCB_4}(T) &=& \sum_{k \in \sP(N), k \neq k^{**}} \Delta_{k} \sum_{i=1}^{M}  \bbE[\tilde{n}_{i,k_{i}}(T)] + C \bbE[m(T)] + \tilde{K}_{\mathcal{X}, P}  \\
& \leq &  \Delta_{max} \sum_{k \in \sP(N), k \neq k^{**}} \sum_{i=1}^{M} \bbE[\tilde{n}_{i,k_{i}}(T)] + C \bbE[m(T)] + \tilde{K}_{\mathcal{X}, P}  
\end{eqnarray*}\begin{eqnarray*}& =& \Delta_{max} \bbE[\tilde{n}(T)] + C \bbE[m(T)] + \tilde{K}_{\mathcal{X}, P} .
\end{eqnarray*}
For a general $L$, by Theorem \ref{thm:ucb1} 
\begin{eqnarray*}
\tilde{\sR}_{M, \tt dUCB_4}(T)  &\leq & L \Delta_{max} \bbE[\tilde{n}(T)] + C(f(L)) \bbE[m(T)] + \tilde{K}_{\mathcal{X}, P} . \\
& \leq & (L \Delta_{max} + C(f(L))  (1+ \log T)) \bbE[\tilde{n}(T)] + \tilde{K}_{\mathcal{X}, P} .
\end{eqnarray*}
 Now, using the bound (\ref{dUCB_subopt_M}), we get the desired upper bound on the expected regret. \\
(ii) This can now easily be obtained using the above and following Theorem \ref{thm:ducb4}.
\end{proof}


\section{Distributed Bipartite Matching: Algorithm and Implementation}\label{sec:dbm}

In the previous section, we referred to an unspecified distributed algorithm for bipartite matching ${\tt dBM}$, that is used by the ${\tt dUCB_4}$ algorithm. We now present one such algorithm, namely, Bertsekas' auction algorithm \cite{Be92}, and its distributed implementation. We note that the presented algorithm is not the only one that can be used. The ${\tt dUCB_4}$ algorithm will work with a distributed implementation of any bipartite matching algorithm, e.g. algorithms given in \cite{ZaSpPa08}. 

Consider a bipartite graph with $M$ players on one side, and $N$ arms on the other, and $M\leq N$. Each player $i$ has a value $\mu_{i,j}$ for each arm $j$. Each player knows only his own values. Let us denote by $k^{**}$, a matching that maximizes the matching surplus $\sum_{i,j} \mu_{i,j}x_{i,j}$, where the variable $x_{i,j}$ is 1 if $i$ is matched with $j$, and 0 otherwise. Note that $\sum_i x_{i,j} \leq 1, \forall j$, and $\sum_j x_{i,j} \leq 1, \forall i$. Our goal is to find an $\epsilon$-optimal matching. We call any matching $k^*$ to be $\epsilon$-optimal if $\sum_{i} \mu_{i,k^{**}(i)} - \sum_{i} \mu_{i,k^*(i)} \leq \epsilon$.


\begin{algorithm}{}
\caption{: ${\tt dBM_{\e}}$ ( Bertsekas Auction Algorithm)}
\label{algo:dBM}
{\small \begin{algorithmic}[1]
\STATE All players $i$ initialize prices $p_j = 0, \forall ~\text{channels} ~j$;
\WHILE{(prices change)}
\STATE Player $i$ communicates his preferred arm $j_i^*$ and bid $b_i = \max_j (\mu_{ij}-p_j) - \text{2max}_j(\mu_{ij}-p_j) + \frac{\epsilon}{M}$ to all other players.
\STATE Each player determines on his own if he is the \textit{winner} $i_j^*$ on arm $j$;
\STATE All players set prices $p_j = \mu_{i_j^*,j}$;
\ENDWHILE
\end{algorithmic}
}
\end{algorithm}

Here, $\text{2max}_j$ is the second highest maximum over all $j$. The \textit{best} arm for a player $i$ is arm $j_i^* = \arg \max_j (\mu_{i,j}-p_j)$. The \textit{winner} $i_j^*$ on an arm $j$ is the one with the highest bid.  

The following lemma in \cite{Be92} establishes that Bertsekas' auction algorithm will find the $\e$-optimal matching in a finite number of steps. 

\begin{lemma}\cite{Be92}
Given $\epsilon>0$, Algorithm~\ref{algo:dBM} with rewards $\mu_{i,j}$, for player $i$ playing the $j$th arm, converges to a matching $k^*$ such that $\sum_{i} \mu_{i, k^{**}(i)} - \sum_{i} \mu_{i, k^*(i)} \leq \epsilon$ where $k^{**}$ is an optimal matching. Furthermore, this convergence occurs in less than $(M^2\max_{i,j}\{\mu_{i,j}\})/\epsilon$ iterations.
\end{lemma}


The temporal structure of the ${\tt dUCB_4}$ algorithm is such that time is divided into frames of length $L$. Each frame is either a \textit{decision} frame, or an \textit{exploitation} frame. In the exploitation frame, each player plays the arm it was allocated in the last decision frame.  The distributed bipartite matching algorithm (e.g. based on Algorithm \ref{algo:dBM}), is run in the decision frame. The decision frame has an interrupt phase of length $M$ and negotiation phase of length $L-M$. We now describe an implementation structure for these phases in the decision frame.

\begin{figure}[tbh]
\centering{\includegraphics[scale=0.4]{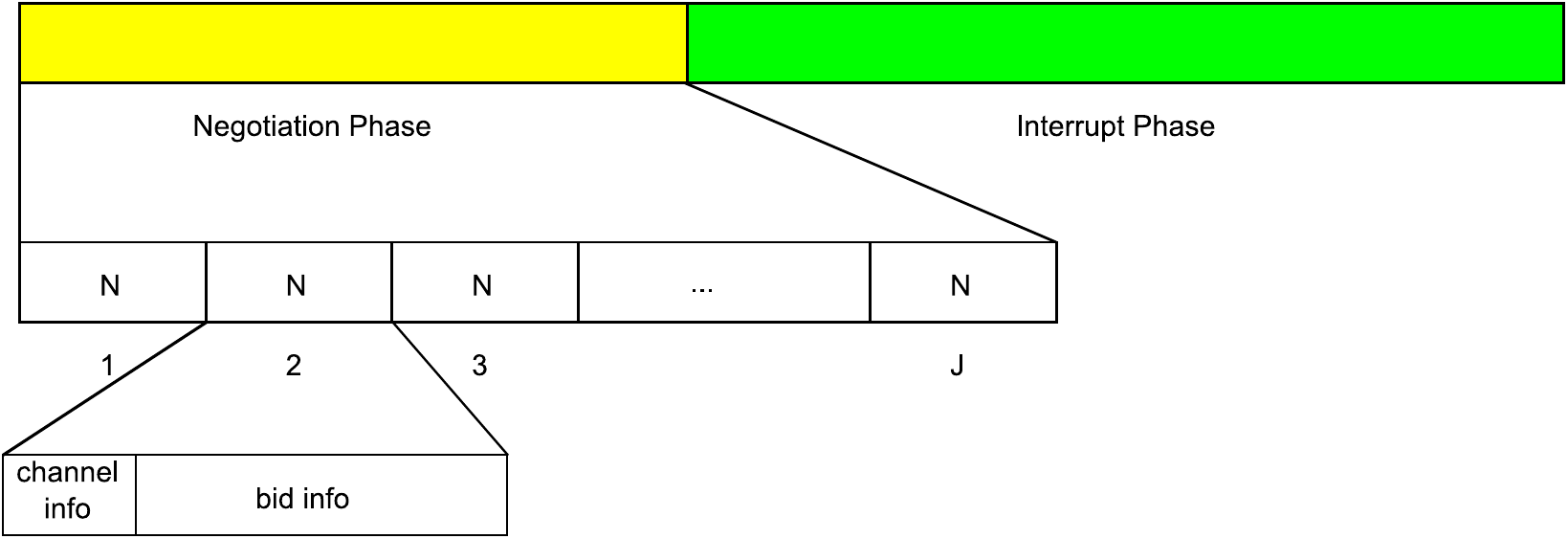}}
\caption{Structure of the decision frame}
\label{fig:negphasephy}
\end{figure}
\noindent \textbf{Interrupt Phase:}  The interrupt phase can be implemented very easily. It has length $M$ time slots. On a pre-determined channel, each player by turn transmits a `1' if the arm with which it is now matched has changed, `0' otherwise. If any user transmits a `1', everyone knows that the matching has changed, and they reset their counter $\eta=1$.

\noindent \textbf{Negotiation Phase:} The information needed to be exchanged to compute an $\e$-optimal matching is done in the negotiation phase. We first provide a \textit{packetized implementation} of the negotiation phase. The negotiation phase consists of $J$ subframes of length $M$ each (See figure \ref{fig:negphasephy}). In each subframe, the users transmit a packet by turn. The packet contains bid information: (channel number, bid value). Since all users transmit by turn, all the users know the bid values by the end of the subframe, and can compute the new allocation, and the prices independently. The length of the subframe $J$ determines the precision $\epsilon$ of the distributed bipartite matching algorithm. Note that in the packetized implementation, $\e_1=0$, i.e., bid values can be computed exactly, and for a given $\e_2$, we can determine $J$, the number of rounds the ${\tt dBM}$ algorithm \ref{algo:dBM} runs for, and returns an $\e_2$-optimal matching. 

If a packetized implementation is not possible, we can give a \textit{physical implementation}. Our only assumption here is going to be that each user can observe a channel, and determine if there was a successful transmission on it, a collision, or no transmission, in a given time slot. The whole negotiation phase is again divided into $J$ sub-frames. In each sub-frame, each user transmits by turn. It simply transmits $\lceil{\log M} \rceil $ bits to indicate a channel number, and then $\lceil{\log 1/\e_1}\rceil$ bits to indicate its bid value to precision $\e_1$. The number of such sub-frames $J$ is again chosen so that the ${\tt dBM}$ algorithm (based on Algorithm \ref{algo:dBM}) returns an $\e_2$-optimal matching.

\section{Simulations}
\label{sec:simulations}
\begin{figure}[b]
\centering{
\hspace*{-10mm}\includegraphics[width=4in]{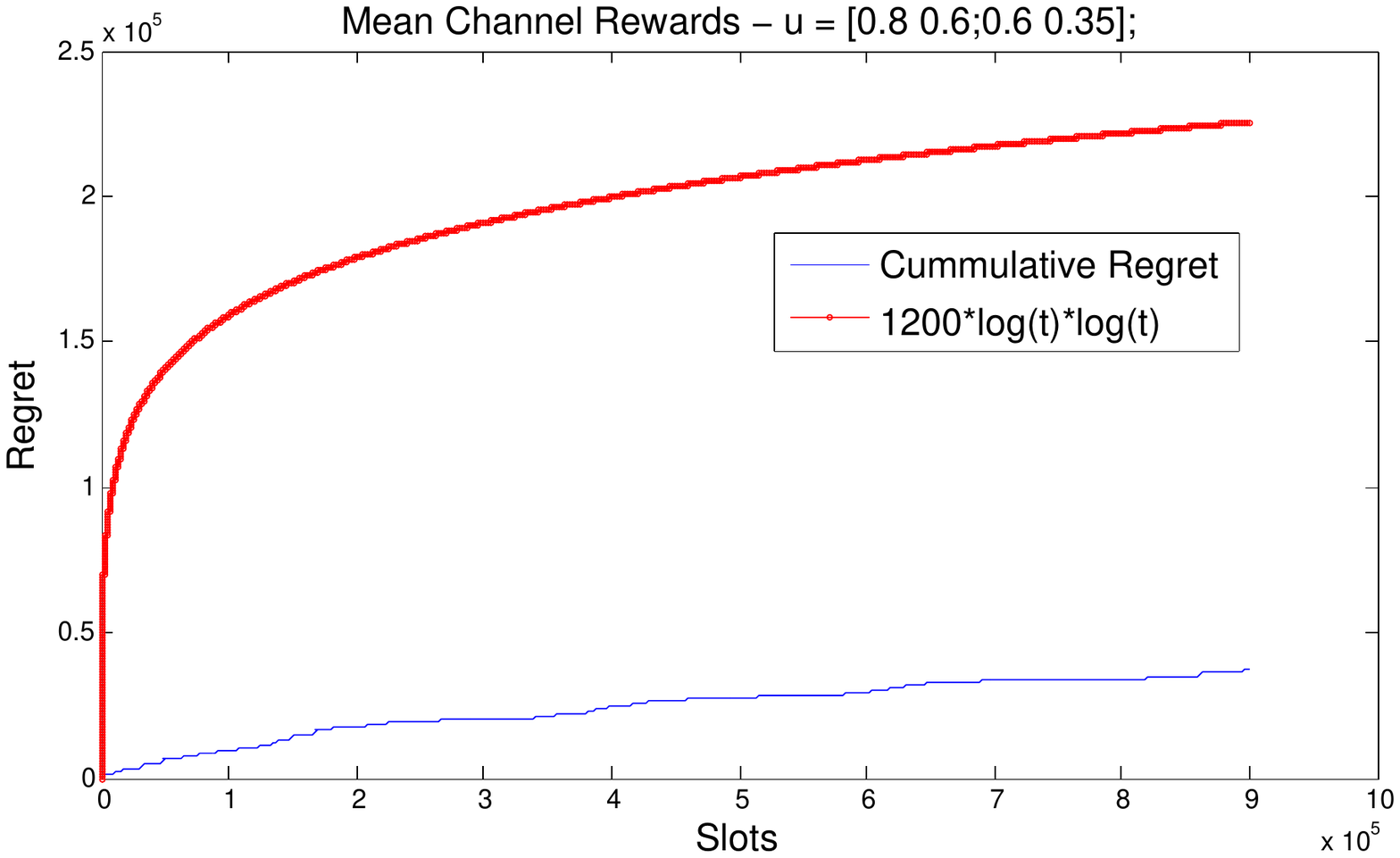}\includegraphics[ width=4.1in]{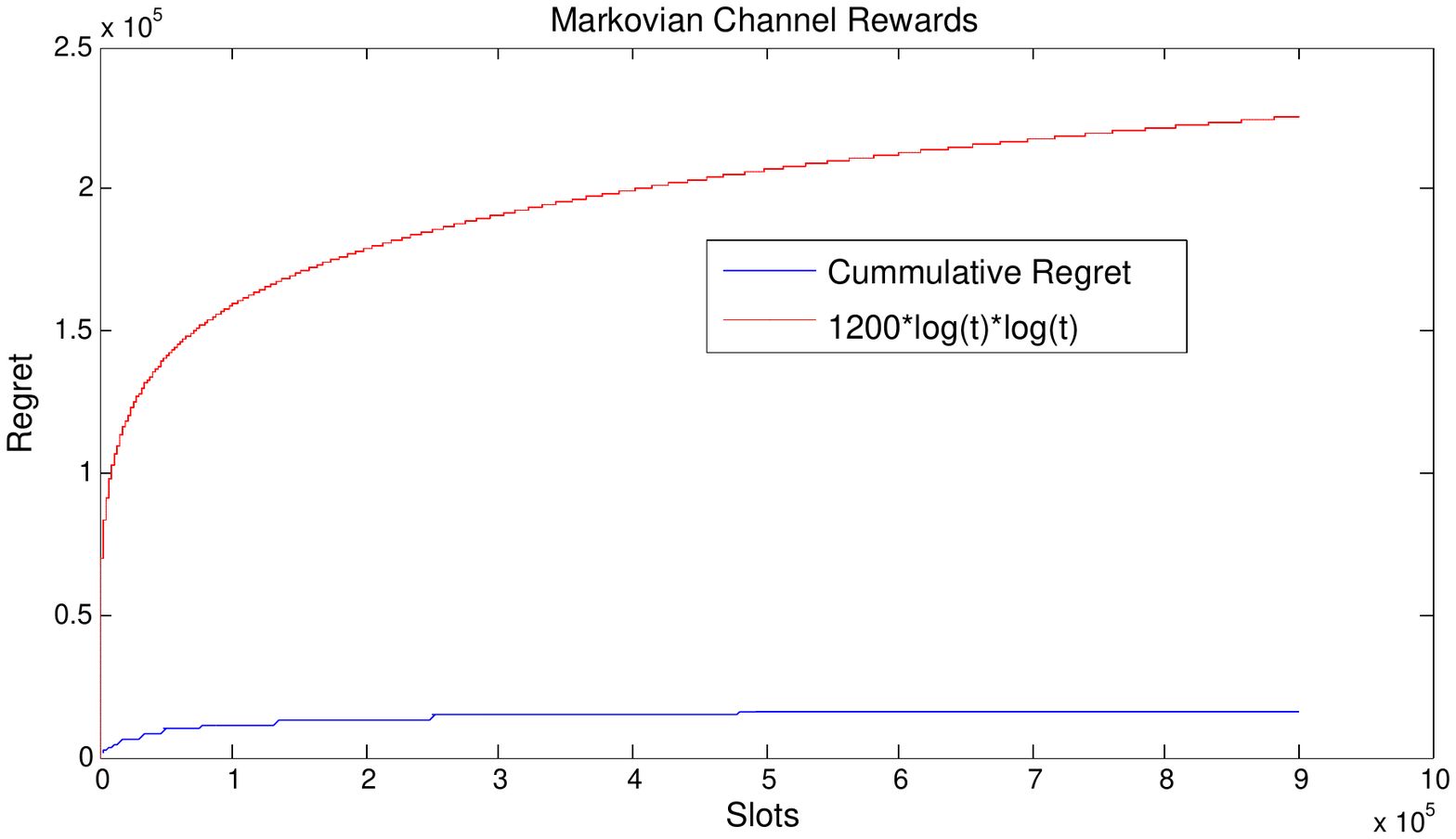}}
\caption{(i) Cumulative regret : $2$ users, $2$ channels; i.i.d. channels; Mean reward matrix = $[0.8, 0.6; 0.6, 0.35]$.
(ii) Cumulative regret : $2$ users, $2$ channels; Markovian channels.}
\label{fig:decentralized-2}
\end{figure}
We illustrate the empirical performance of the ${\tt dUCB_4}$ algorithm when the successive rewards from a channel are  i.i.d. and when they are Markovian. Consider two users and two channels. In the i.i.d. case, each channel has rewards that are generated by a Bernoulli distribution taking values $0$ and $1$. The first user has mean rewards of $0.8$ and $0.6$ for channels $1$ and $2$ respectively. The second user has mean rewards of $0.6$ and $0.35$. The algorithm's performance, averaged over 50 runs, is shown in Figure~\ref{fig:decentralized-2} (i). It shows cumulative regret with time. The red bold curve is the theoretical upper bound we derived, while the blue curve is the observed regret. The algorithm seems to perform much better than even the poly-log regret upper bound we derived.

In the Markovian case, rewards are generated by a Markov chain having states $0$ and $1$. The mean reward on a  channel is given by its stationary distribution, i.e., the probability the Markov chain is in state $1$, $\pi_1$. The properties of the Markov chains are given in Table~\ref{tab:sim-markov-example}. The performance of the ${\tt dUCB_4}$ algorithm on this model, averaged over 50 runs, is shown in Figure~\ref{fig:decentralized-2} (ii). Once again, the algorithm seems to perform much better than even the poly-log regret upper bound we derived.

\begin{table}
\caption{Markov Chain Parameters : Transition probability and Stationary distribution}
\begin{center}
\begin{tabular}{| c | c | c | c |}\hline
\label{tab:sim-markov-example}
User & Channel & $p_{01}$,$p_{10}$ & $\pi$ \\\hline
1 & 1 & $0.3$,$0.5$ & $0.3/0.8$\\\hline
1 & 2 & $0.2$,$0.6$ & $0.2/0.8$\\\hline
2 & 1 & $0.6$,$0.3$ & $0.6/0.9$\\\hline
2 & 2 & $0.7$,$0.2$ & $0.7/0.9$\\\hline
\end{tabular}
\end{center}
\end{table}

\section{Conclusions}\label{sec:conclusions}
We have proposed a ${\tt dUCB_4}$ algorithm for decentralized learning in multi-armed bandit problems that achieves a regret of near-$O(\log^{2}(T))$. Finding a lower bound is usually quite difficult, and currently a work in progress.  

\bibliographystyle{IEEEtran}
\bibliography{refs-dmab}

\appendix

Let $\left(X_{t}, t=1, 2, \ldots \right)$ be an irreducible, aperiodic and reversible Markov chain on a finite state space $\mathcal{X}$ with  transition probability matrix $P$, a stationary distribution $\pi$ and  an initial distribution $\lambda$.  Let $\mathcal{F}_{t}$ be the $\sigma$-algebra generated by $  \left(X_{1}, X_{2}, \ldots, X_{t}\right)$. Denote $N_{\lambda} =  \left \lVert \left( \frac{\lambda_{x}}{\pi_{x}}, x \in \mathcal{X} \right) \right \rVert_{2}$. 

\begin{lemma}\cite{AnVaWa87b}
\label{lemma_AnVa1}
Let $\mathcal{G}$ be a $\sigma$-algebra independent of $\mathcal{F} = \vee_{t \geq 1 } F_{t}$. Let $\tau$ be a stopping time of $\mathcal{F}_{t} \vee \mathcal{G}$. Let $N(x, \tau) := \sum_{t=1}^{\tau} I\{X_{t}=x\}  $. Then, 
\(
| \bbE[N(x, \tau)] - \pi_{x} \bbE[\tau] | \leq K,
\)
where $K \leq 1/\pi_{min} $ and $\pi_{min}=\min_{x \in \mathcal{X} } \pi_{x}$. $K$ depends on $P$.
\end{lemma}

\begin{lemma}\cite{Le98}
\label{lemma_Le}
Denote $N_{\lambda} =  \left \lVert \left( \frac{\lambda_{x}}{\pi_{x}}, x \in \mathcal{X} \right) \right \rVert_{2}$. Let $\rho$ be the  eigenvalue gap, $1 - \lambda_{2}$, where $\lambda_{2}$ is the second largest eigenvalue of the matrix $P^{2}$. Let $f:\mathcal{X} \rightarrow \mathbf{R}$ be such that $\sum_{x \in \mathcal{X} } \pi_{x} f(x) = 0$, $\left \lVert f \right \rVert_{\infty} \leq 1, {\left \lVert f \right \rVert}^{2}_{2} \leq 1$. Then, for any $ \gamma > 0$, 
\(\mathbb{P}\left( \sum_{a=1}^{t} f(X_{a})/t \geq \gamma \right) \leq N_{\lambda} e^{-t\rho \gamma^{2} /28}.\)
\end{lemma}

\end{document}